\documentclass{amsart}

\usepackage{amssymb,amsmath,amscd, amsthm,xspace,color,tocvsec2, array, 
	tikz-cd, aligned-overset}
\usepackage[all, cmtip]{xy}
\usepackage[mathscr]{euscript}
\usepackage{enumerate}
\usepackage[breaklinks=true]{hyperref}
\usepackage[letterpaper]{geometry}
\usepackage[all, cmtip]{xy}
\usepackage[ampersand]{easylist}
\usepackage{stmaryrd}
\usepackage[myheadings]{fullpage}
\usepackage{csquotes}
\usepackage{mathtools}

	\newcommand{\fk}{\mathfrak{k}}

\newcommand{\fkk}{\fk_\kappa}

\newcommand{\sU}{U}

\newcommand{\HCh}{\on{HCh}}
\newcommand{\HChl}{\on{HCh}^{\on{tame}}}

\newcommand{\per}{\Pos_{\hspace{.2mm} \kappa}}
\newcommand{\hk}{\widehat{\fh}_\kappa}

\newcommand{\Pos}{\on{Pos}}
\newcommand{\PosO}{\fh[\hspace{-.2mm}[z]\hspace{-.2mm}]\mod^{K, \Lambda, \heartsuit}}

\newcommand{\sS}{\mathscr{S}}

\newcommand{\HK}{\sH}

\newcommand{\sH}{\mathscr{H}}

\newcommand{\St}{\on{s}}

\newcommand{\Hom}{\operatorname{Hom}}

\newcommand{\sN}{\mathscr{N}}

\newcommand{\g}{\widehat{\mathfrak{g}}}

\newcommand{\fl}{\mathfrak{l}}

\newcommand{\halpha}{\check{\alpha}}

\newcommand{\fsl}{\mathfrak{sl}}

\newcommand{\ft}{\mathfrak{t}}

\newcommand{\id}{\text{id}}

\newcommand{\fg}{\mathfrak{g}}

\newcommand{\comment}[1]{}

\newcommand{\OO}{\mathscr{O}}

\newcommand{\fb}{\mathfrak{b}}
\newcommand{\gk}{\g_\kappa}
\newcommand{\fn}{\mathfrak{n}}

\newcommand{\sD}{\mathscr{D}}

\newcommand{\on}{\operatorname}

\newcommand{\scc}{\mathscr{C}}

\newtheorem{theo}[subsubsection]{Theorem}
\newtheorem*{theo*}{Theorem}
\newtheorem{lem}[subsubsection]{Lemma}
\newtheorem{lemma}[subsubsection]{Lemma}
\newtheorem{cor}[subsubsection]{Corollary}
\newtheorem{pro}[subsubsection]{Proposition}

\theoremstyle{remark}
\newtheorem{defn}[subsubsection]{Definition}
\newtheorem{re}[subsubsection]{Remark}

\newtheorem{ex}[subsubsection]{Example}

\numberwithin{equation}{section}

\newcommand{\fh}{\mathfrak{h}}
\newcommand{\sC}{\mathscr{C}}
\newcommand{\sM}{\mathscr{M}}
\newcommand{\sHom}{\on{Hom}}
\newcommand{\inv}{\msf{inv}}
\newcommand{\msf}{\operatorname}

\newcommand{\sW}{\EuScript{W}}

\newcommand{\Vect}{\on{Vect}}

\renewcommand{\mod}{\operatorname{-mod}}

\newcommand{\DGCat}{\msf{DGCat}_{\on{cont}}}

\begin{document}


\title[]{Affine Harish-Chandra bimodules and  Steinberg--Whittaker localization}

\author{Justin Campbell} 
\address{ California Institute of Technology \\ Pasadena, CA 91125, USA}

\email{jcampbell@caltech.edu}

\author{Gurbir Dhillon}
\address{Yale University \\ New Haven, CT
06511, USA}

\email{gurbir.dhillon@yale.edu}

\date{\today}

\begin{abstract} We construct categories of Harish-Chandra bimodules for affine Lie algebras analogous to Harish-Chandra bimodules with infinitesimal characters for simple Lie algebras, addressing an old problem raised by I. Frenkel and Malikov. Under an integrality hypothesis, we establish monoidal equivalences between blocks of our affine Harish-Chandra bimodules and versions of the affine Hecke category. To do so, we introduce a new singular localization onto Whittaker flag manifolds. The argument for the latter  specializes to a somewhat alternative treatment of regular localization, and  identifies the category of Lie algebra representations with a generalized infinitesimal character as the parabolic induction, in the sense of categorical representation theory, of a Steinberg module.   \end{abstract}

\maketitle

\setcounter{tocdepth}{1}
\tableofcontents

\section{Introduction}

\subsection{} The study of Harish-Chandra bimodules with generalized infinitesimal characters has been of central importance in the representation theory of semisimple Lie algebras. As an early triumph, it identified categories of principal series representations of complex Lie groups with categories of highest weight representations of semisimple Lie algebras  \cite{dixmier}, \cite{duflo}, \cite{joseph}, \cite{begel}, and thereafter with Hecke categories and Soergel bimodules \cite{beilinson-bernstein}, \cite{soergelhc}, \cite{beilinsonginzburg}. 

The representation theory of affine Lie algebras is by now a mature subject, with well understood theories of highest weight representations, affine Hecke categories, and affine Soergel bimodules. However, a corresponding theory of affine Harish-Chandra bimodules has been elusive. This absence has been felt by multiple groups of experts for at least a quarter century, notably in the work and conjectures of I. Frenkel and Malikov \cite{frenkelmalikov1}, \cite{frenkelmalikov2}. 

\subsection{} In this paper we introduce such a category of affine Harish-Chandra bimodules, and in the process exhibit a partial substitute for infinitesimal characters for Kac--Moody Lie algebras. To show our category of bimodules has the sought-for properties, we provide a new type of singular localization onto Whittaker flag manifolds. Our method of proof for localization, which uses techniques from categorical representation theory, lifts the fundamental identification of the Grothendieck group of a block of Category $\OO$ with an induced representation of the Hecke algebra to an equivalence of the block itself with an induced representation of the Hecke category.

In the remainder of the introduction, we describe this all in more detail; after giving the definition  of the category of bimodules considered in Section \ref{ss:introbimods}, we then discuss our main Theorems \ref{t:hch=hecke} and \ref{t:locintro}.

\subsection{Affine Harish-Chandra bimodules}
\label{ss:introbimods}
\subsubsection{} There are two basic difficulties in giving an analogue of Harish-Chandra bimodules with generalized infinitesimal characters in the affine setting. We now review these, and explain the approach taken in the present work to addressing them.

\subsubsection{} The first problem is to give a good notion of a general affine Harish-Chandra bimodule. To see the difficulty,  fix an almost simple, simply connected algebraic group $G$ with Lie algebra $\fg$. Recall that a Harish-Chandra bimodule is a $\fg \oplus \fg$-module on which the diagonal action of $\fg$ integrates to an action of $G$. 

Let us now pass to the loop group $G(\!(z)\!)$, whose Lie algebra is the loop algebra $\fg(\!(z)\!)$. One may again consider the abelian category of smooth $\fg(\!(z)\!) \oplus \fg(\!(z)\!)$-modules on which the diagonal action is integrable. However, as $G(\!(z)\!)$ has no nontrivial smooth representations, a short argument shows that this implies that all of $\fg(\!(z)\!) \oplus \fg(\!(z)\!)$ acts trivially. So the obtained category is equivalent to vector spaces, and is too small to be of use. 
 
When one passes from the loop algebra to an affine Lie algebra $\gk$, i.e. one introduces the central charge of level $\kappa$, problems of a similar nature arise. 
 
\subsubsection{} To obtain more objects, we follow a  proposal of Gaitsgory \cite{quantum-langlands-summary}. Namely, if one tries to produce more objects by passing to the derived category of bimodules, one encounters the same degeneracy as in the abelian case. However, E. Frenkel and Gaitsgory introduced in \cite{fg} a certain enlargment of the derived category, which we denote below by %
\begin{equation} \label{e:bigcat}
     \fg(\!(z)\!) \oplus \fg(\!(z)\!)\mod.
\end{equation}
This dg-category, sometimes called the {\em renormalized} derived category, contains the usual unbounded derived category as a full subcategory, and carries a $t$-structure for which the bounded below objects agree with the usual bounded below derived category. Said informally, one adjoins objects concentrated in cohomological degree $-\infty$. The role of such objects and such a category had been long anticipated in the rich folklore around semi-infinite mathematics, with important contributions by Arkhipov, Feigin, Finkelberg, E. Frenkel,  Positselski, Voronov, and others; see \cite{vor}, \cite{pos} for a partial indication.\footnote{As a striking example, the reader may wish to compare Voronov's final remarks in \cite{vor} with Krause's renormalization of the category of quasi-coherent sheaves on a Noetherian scheme \cite{krause}, i.e. the category of ind-coherent sheaves \cite{indcoh}.}  

 One then takes the category of affine Harish-Chandra bimodules to be the $G(\!(z)\!)$-equivariant objects in the enlarged category \eqref{e:bigcat}. To make sense of this requires some categorical representation theory of groups. Briefly,  the diagonal adjoint action of $G(\!(z)\!)$ on $\fg(\!(z)\!) \oplus \fg(\!(z)\!)$ induces an action of the category of D-modules $D( G(\!(z)\!))$, viewed as a monoidal dg-category under convolution, on \eqref{e:bigcat}. A general procedure then attaches to any cocomplete dg-category $\sC$ equipped with an action of $D(G(\!(z)\!))$ its category of equivariant objects. If one runs the above instead for the finite dimensional group $G$, this recovers the usual (derived) category of Harish-Chandra bimodules.

\subsubsection{} At the time of Gaitsgory's proposal, the formalism for passing to equivariant objects was not yet available. Plainly, one would like to compute such a category as the totalization of a cobar resolution akin to the simplicial presentation of the equivariant derived category of sheaves on a quotient stack. However, this uses inverse limits of categories, which cannot be sensibly defined working at the triangulated level. The requisite categorical representation theory has been developed over the ensuing fifteen years by Gaitsgory, with significant contributions by Beraldo, Ben-Zvi--Nadler, and Raskin, and building on earlier work of Beilinson and Drinfeld \cite{bdh}, \cite{fg2}, \cite{bznadler}, \cite{beraldo}, \cite{whit}, \cite{mys}, \cite{whitlocglob}. 

The above category of affine Harish-Chandra bimodules, and its analogues at other central charges, figures prominently in the main conjectures of the local quantum Langlands correspondence, cf. \cite{quantum-langlands-summary}, \cite{winter-langlands-summary}. To our knowledge the present work is the first to undertake its study.

\subsubsection{} We now turn to the second basic difficulty. Recall that in finite type to obtain categories of Harish-Chandra bimodules equivalent to Hecke categories and categories of highest weight representations, one  fixes a strict or generalized infinitesimal character for the underlying $\fg\oplus\fg$-module.


This forces us to confront a fundamental issue in the representation theory of affine algebras, which experts have contended with since the early days of the subject. Namely, the center of the completed enveloping algebra $\sU(\hspace{.5mm}\gk)$ is {trivial} for all levels $\kappa$ excepting the critical level $\kappa_c$. Nonetheless, it has long been apparent that its category of modules behaves in certain respects as though there is a nontrivial center. Indeed, in the literature one finds various categories of representations whose linkage principles all take the same form; see \cite{dgk}, \cite{lianzuckerman}, \cite{yakimov}, \cite{adamovic}. Namely, each time at a noncritical level $\kappa$ they are indexed by orbits in the dual abstract Cartan $\ft^*$ of the level $\kappa$ dot action of the affine Weyl group.

So, we must look for an alternative way to isolate, for any noncritical level $\kappa$ and affine Weyl group orbit $\chi$, a full subcategory of modules
\begin{equation} \label{e:subcatl}
     \gk\mod_\chi \subset \gk\mod,
\end{equation}
which acts as though it were the category of modules with `generalized infinitesimal character $\chi$'. 

\subsubsection{} Our approach to this second problem, which is perhaps the basic new idea of the present work, is again via categorical representation theory. 

Recall that, given an algebra $A$ and a module $M$, from any element $m$ of $M$ one can form the $A$-submodule  generated by $m$. Similarly, given a monoidal category and a module category, to any object of the latter one can associate the submodule generated by it.

In this way, we {define} \eqref{e:subcatl} to be the $D_\kappa(G(\!(z)\!))$-submodule of $\gk\mod$ generated by the Verma module of highest weight $\lambda$, for any element $\lambda$ of the orbit $\chi$; the resulting submodule is independent of the choice. If one runs this definition for the finite dimensional group $G$, it recovers the usual derived category of $\fg$-modules with generalized infinitesimal character that of the corresponding Verma module. 
So, {\em we use the categorical action to bypass the absence of infinitesimal characters. }

Combining this and the construction \eqref{e:bigcat}, we obtain for appropriate pairs of weights $\chi, \phi$ and levels $\kappa, \kappa_\circ$ the desired categories of affine Harish-Chandra bimodules.\footnote{The term `appropriate' here means that, just as in finite type, the difference between $(\chi, \kappa)$ and $(\phi, \kappa_\circ)$ should be integral (or the categories would vanish). } Namely, we again pass to $G(\!(z)\!)$-equivariant objects to obtain the submodule
\begin{equation} \label{e:hchwts}
    \on{HCh}_{\chi, \phi} := (\gk\mod_\chi \otimes \hspace{.5mm} \widehat{\fg}_{\kappa_\circ}\mod_\phi)^{G(\!(z)\!)} \subset \gk \oplus \widehat{\fg}_{\kappa_\circ}\mod^{G(\!(z)\!)}. 
\end{equation}

\subsubsection{} We spend the majority of the paper verifying that the categories of Harish-Chandra bimodules \eqref{e:hchwts} introduced above behave in the desired ways. This is not completely transparent, since in the most important cases every single object of $\on{HCh}_{\chi,\phi}$ is a homological phantom, i.e. is concentrated in cohomological degree $-\infty$.

\subsubsection{} Before discussing the bimodules, let us briefly describe some features of the categories of $\gk$-modules appearing in \eqref{e:subcatl}. As $\chi$ runs through the set of affine Weyl group orbits, the subcategories \eqref{e:subcatl} contain all parabolic inductions of modules with infinitesimal characters from proper Levi factors of $\gk$. Moreover, they satisfy the linkage principle 
\begin{equation} \label{e:incgmodchi}
    \underset{\chi} \oplus \hspace{1mm} \gk\mod_\chi \subset \gk\mod,
\end{equation}
i.e. for distinct orbits $\chi_1, \chi_2$ the categories $\gk\mod_{\chi_1}$ and $\gk\mod_{\chi_2}$ are left and right orthogonal to one another. Combining these points confirms a linkage principle for positive energy representations predicted by Yakimov \cite{yakimov}, cf. Section \ref{s:linky} for more details.  

If $\kappa$ is an integral level, under the (conjectural) local geometric Langlands correspondence, one can associate to $\gk\mod$ a sheaf of categories $\sS$ on the moduli space of $\check{G}$-local systems on the punctured disk, where $\check{G}$ denotes the Langlands dual of $G$. The inclusion \eqref{e:incgmodchi} should correspond to the subcategories of $\sS$ supported on formal neighborhoods of local systems with regular singularities. In particular, the inclusion \eqref{e:incgmodchi} does not see the wildly ramified part of $\gk\mod$. However, for our purposes, namely to have the sought-for relations with affine Hecke categories, this is a necessary feature, and not a bug. We turn to these relations next.

\subsection{Equivalences with affine Hecke categories}

\subsubsection{} In the next theorem  we identify our categories of bimodules \eqref{e:hchwts} with variants of the affine Hecke category under an integrality assumption. This result strongly suggests that our definition is the correct one. 

 To state it, we need some notation. Let $\kappa, \kappa_\circ$ be noncritical integral levels. At such levels, the dot action of the affine Weyl group preserves the weight lattice. Denote by $\chi$ and $\phi$ a pair of orbits of weights at levels $\kappa$ and $\kappa_\circ$, respectively. Fix a pinning of the loop group, and in particular an Iwahori subgroup 
\[
   I \subset G(\!(z)\!),
\]
whose prounipotent radical we denote by $\mathring{I} \subset I$. By a standard procedure, one may then associate to $\chi$ and $\phi$ character sheaves on $\mathring{I}$, which we denote by $\psi$ and $\omega$, respectively. Loosely, as the orbit $\chi$ grows more singular the character $\psi$ grows more nondegenerate, cf. Section \ref{ss:parindst} for the details.  

We also need one more definition, which may be new. For a character sheaf $\psi$ as above, consider the corresponding category of Iwahori--Whittaker D-modules on the loop group  
\begin{equation} \label{e:iwcat}
       D\big(G(\!(z)\!) /\mathring{I}, -\psi\big). 
\end{equation}
This is a naturally a $D(G(\!(z)\!))$-module under left convolution. We may therefore consider the submodule generated by the Iwahori--equivariant objects. This is a categorical analogue of the (degenerate) Steinberg module of a finite group of Lie type, and is a colocalization of \eqref{e:iwcat}, i.e. its inclusion admits a right adjoint. 

Since \eqref{e:iwcat} corepresents the functor of taking the Iwahori--Whittaker equivariant objects $\sC^{\mathring{I}, \psi}$ of a categorical representation $\sC$ of the loop group, its Steinberg submodule determines a corresponding full subcategory of {\em Steinberg--Whittaker} equivariant objects 
\[
   \sC^{\mathring{I}, \psi, \St} \subset \sC^{\mathring{I}, \psi}. 
\]
For example, if $\psi$ is the trivial character sheaf, the Steinberg--Whittaker invariants are the Iwahori monodromic objects of $\sC$. Finally, we note that the Steinberg--Whittaker invariants are also introduced in forthcoming work of Arinkin and Bezrukavnikov. 

With these preparations, we may state our basic theorem on affine Harish-Chandra bimodules.

\begin{theo}\label{t:hch=hecke} Let $\chi, \phi$ and $\psi, \omega$ be as above. There are canonical equivalences
\[
   \on{HCh}_{\chi, \phi} \simeq D\big(\hspace{.3mm}(\mathring{I}, \psi, \St \hspace{0mm})\backslash \hspace{.1mm}G(\!(z)\!) \hspace{.1mm}/ \hspace{.3mm}(\mathring{I}, -\omega, \St) \hspace{.3mm}\big) \simeq \widehat{\fg}_{\kappa_\circ}\mod_\phi^{\mathring{I}, \psi}. 
\]
\end{theo}

We now make a few remarks about Theorem \ref{t:hch=hecke}. 
\begin{re} The analogues of the above equivalences for the finite dimensional group $G$ are fundamental results, due in various amounts of generality to Bernstein--Gelfand, Beilinson--Bernstein, Milicic--Soergel, and more recently Ginzburg, Webster, and Losev \cite{begel}, \cite{beilinson-bernstein}, \cite{ms}, \cite{ginz}, \cite{bw}, \cite{los}. However, as far as we are aware, in the case of $\chi$ and $\phi$ both singular, it was only a folklore expectation that there should be a realization as appropriate bi-Whittaker sheaves.

For context, we note that the analogous categories of bi--Whittaker sheaves of general Kac--Moody groups should be Koszul dual to bi-parabolic Hecke categories, and admit equivalences with singular Soergel bimodules. 

We should also mention that in the sheaf--theoretic realization, the Steinberg condition on the Whittaker invariants appearing on one side in fact implies it holds on the other. For this reason, in the literature one does not find the Steinberg--Whittaker invariants explicitly introduced except in the case of the trivial character.

Finally, we emphasize that, to us, the main point of Theorem \ref{t:hch=hecke} is that our category of affine Harish-Chandra bimodules is really the  sought-for one, i.e., it behaves like its finite dimensional counterpart. Unfortunately, to prove that this is so, we cannot mimic the existing arguments. Namely, much of the above work relies on favorable properties of Harish-Chandra bimodules with infinitesimal characters. We only obtain their affine analogues {\em a posteriori} as consequences of the theorem.

\end{re}  

\begin{re} The bimodule and sheaf-theoretic categories appearing above have natural convolution structures. Namely, given affine Harish-Chandra bimodules $\sM_1$ and $\sM_2$ at appropriate levels, one may compose them by forming the semi-infinite homology 
\[
     \sM_1 \underset{\widehat{\fg}}{\overset{ \frac{\infty}{2}}{\otimes}}\hspace{.5mm} \sM_2.  
\]
By construction the equivalence of Theorem \ref{t:hch=hecke} exchanges this with the convolution of the corresponding D-modules on the loop group, cf. Remark \ref{r:monoidaleq} for the details.  \end{re}

\begin{re} We also introduce the analogues of categories with {strict} infinitesimal characters. To see how, recall that in finite type, the category $\fg\mod_\lambda$ is acted on by a completion of $Z(\fg)$, i.e. the log of monodromy. We extend this to affine type by recognizing this as the {\em equivariant Hochschild cohomology} of $\fg\mod_\lambda$. We may then set the monodromy to zero, cf. Section \ref{s:someapps} for the details.
\end{re}

\subsection{Theorem \ref{t:hch=hecke} via Steinberg--Whittaker localization}

\subsubsection{} The idea of the proof of Theorem \ref{t:hch=hecke} is simple. First, some abstract nonsense: if $\sC$ and $\sD$ are categorical loop group representations, such that $\sC$ is {\em dualizable} as an abstract cocomplete dg-category, i.e. as an object of $\DGCat$, we have canonical identifications 
\begin{equation} \label{e:homscats}
     \on{Hom}_{D(G(\!(z)\!))\mod}(\sC, \sD) \simeq \on{Hom}_{\DGCat}(\sC, \sD)^{G(\!(z)\!)} \simeq (\sC^\vee \otimes \sD)^{G(\!(z)\!)},
\end{equation}
cf. Section \ref{s:pn} below for the precise definitions of $D(G(\!(z)\!))\mod$ and $\DGCat$.

We apply this as follows. For Kac--Moody representations, the operation of taking semi-infinite cohomology yields a canonical duality 
\begin{align} 
    & \gk\mod^\vee \simeq \widehat{\fg}_{-\kappa + 2\kappa_c}\mod, 
\intertext{where as before $\kappa_c$ denotes the critical level. Writing $\rho$ for the half sum of the positive roots of $\fg$, this restricts, for any $\lambda \in \ft^*$, to a duality} 
  \label{e:homsgmodl} & \gk\mod_\lambda^\vee \simeq \widehat{\fg}_{-\kappa + 2 \kappa_c}\mod_{-\lambda - 2 \rho}. 
\end{align}
In particular, by combining \eqref{e:homscats} and \eqref{e:homsgmodl}, we recognize our categories of Harish-Chandra bimodules as spaces of intertwining operators between categorical representations.

With this, Theorem \ref{t:hch=hecke} follows formally from rewriting our categories of Kac--Moody representations as categories of D-modules. We do so in the following theorem, which we call {\em Steinberg--Whittaker localization}. 

\begin{theo}\label{t:locintro} Let $\chi, \kappa,$ and $\psi$ be as in Theorem \ref{t:hch=hecke}. Then there is a canonical $D(G(\!(z)\!))$-equivariant equivalence 
\[
    \gk\mod_\chi \simeq D\big( G(\!(z)\!)/ \mathring{I}, \psi, \St \hspace{-.5mm}\big).  
\]

\end{theo}

  Let us discuss the contents of Theorem \ref{t:locintro} in several cases. 
    
   \subsubsection{}  If $\chi$ is regular of negative level, then $\psi$ is trivial, and Theorem \ref{t:locintro} recovers the usual Kashiwara--Tanisaki localization theorem on the enhanced affine flag variety \cite{ktl}.

\subsubsection{} By contrast, if $\chi$ is singular of negative level, to our knowledge the theorem and its analogue for a reductive group $G$ are both new, and are one of the main results of the present work. 

Let us focus on the maximally singular case, e.g. for $G$ with $\chi = - \rho$. Here, the theorem says that localization onto the nondegenerate Whittaker flag manifold yields an equivalence 
\begin{equation}  \label{e:locmaxsing}
   \fg\mod_{-\rho} \simeq D\big(G/N, \psi, \St).
\end{equation}

To get a feel for what is going on, recall that in the regular case, i.e. Beilinson--Bernstein localization, the underlying quasicoherent sheaves of the obtained D-modules on $G/N$ descend to the flag variety $G/B$. This homogeneous space parametrizes maximal unipotent subalgebras $\fn \subset\fg$, and the fibres of the obtained sheaf store (a generalized weight space in) the corresponding $\fn$-homology of the initial $\fg$-module.

In the maximally singular case \eqref{e:locmaxsing}, the underlying quasicoherent sheaves of the obtained D-modules on $G$ descend to $G/Z_GN$, where $Z_G$ denotes the center of $G$. This homogeneous space parametrizes maximal unipotent subalgebras $\fn$ equipped with a generic additive character $\Psi$, and the fibres of the obtained sheaf store the corresponding $\Psi$-twisted $\fn$-homology of the initial $\fg$-module.

 Since the work of Beilinson--Bernstein, the problem of singular localization has been studied by various authors to various ends, with important contributions by Backelin--Kremnizer, Beilinson--Ginzburg, Bezrukavnikov--Mirkovic--Rumynin, and Kashiwara \cite{krem}, \cite{beilinsonginzburg}, \cite{bmr}, \cite{kashloc}.  

To our knowledge \eqref{e:locmaxsing} may be the first nontrivial identification of $\fg\mod_{-\rho}$ with a category of D-modules. Indeed, for any singular generalized infinitesimal character,  localization onto the enhanced flag variety $G/N$ is not fully faithful. A fully faithful localization functor for singular categories was introduced by Backelin--Kremnizer in the interesting work \cite{krem}. As they note, however, for maximally singular blocks it is a special case of the tautological equivalence between all $\fg$-modules and all weakly $G$-equivariant D-modules on $G$.

\subsubsection{}  Finally, unlike at negative level, where one can show the equivalences are $t$-exact, the equivalences for $\chi$ at positive level are of infinite cohomological amplitude, and make essential use of the renormalization of the underlying triangulated categories. They may be considered an affine analogue of usual localization composed with the long intertwining operator.

\subsection{Theorem \ref{t:locintro} via categorical parabolic induction} Finally, we discuss the proof of Steinberg--Whittaker localization.

\subsubsection{} The basic idea is as follows. Many categories in representation theory are well known to `look like' induced representations at the level of Grothendieck groups.

In the example at hand, blocks of Category $\mathscr{O}$ `look like' antispherical modules for the Hecke category, i.e. inductions of sign characters. This point of view has been emphasized by Soergel and his school, and in particular plays a key role in the recent breakthroughs of Williamson and others in the modular representation theory of reductive groups.

As we show in this article, {\em the induction picture of the Grothendieck group underlies an induction at the level of categories}. Given this, the localization theorems are easy consequences. 

\subsubsection{} Let us formulate the contents of the preceding paragraph precisely. The discussion below is for $G$, the case of the loop group being similar. 

  Fix an antidominant weight $\chi$, and write $P$ for the standard parabolic corresponding the stabilizer of $\chi$ under the dot action of the Weyl group. Write $N_P$ for the unipotent radical of $P$, and $L$ for its Levi quotient, with corresponding Lie algebra $\fl$. Note that $\chi$ is a maximally singular weight for $L$. 
  
  Consider the functor of parabolic induction of Lie algebra representations 
  $
     \on{pind}: \fl\mod_\chi \rightarrow \fg\mod_\chi^{N_P}. 
  $
  By adjunction, this induces a functor 
  \begin{equation} \label{e:pindrepz}
          D(G) \underset{D(P)} \otimes \fl\mod_\chi \rightarrow \fg\mod_\chi,
  \end{equation}
  and we show in Theorem \ref{t:gmodparind} that this is an equivalence.

  \begin{re} On $B$-equivariant objects, \eqref{e:pindrepz} yields a canonical equivalence 
  \[
       D(B \backslash G / B) \underset{D(B \backslash P / B)} \otimes \fl\mod_\chi^B \simeq \fg\mod_\chi^B.
  \]
  That is, blocks of Category $\mathscr{O}$ are inductions of sign representations at the level of categories. Similarly, in geometry we find that convolution induces an equivalence 
  \[
      D(B \backslash G / B) \underset{D(B \backslash P / B)} \otimes D(B \backslash P / N, \psi) \simeq D(B \backslash G / N, \psi),
  \]
  where $\psi$ is a generic additive character of $N/N_P$. 
  \end{re}
  
 \subsubsection{} Let us turn to localization theory. In the maximally singular case \eqref{e:locmaxsing}, after introducing a functor between the two sides, by abstract nonsense it suffices to check the obtained map on $B$-equivariant categories is an equivalence. In this case, both sides are equivalent to the category of vector spaces, so this is an easy check. The general case then follows formally by parabolically inducing  \eqref{e:locmaxsing}. A similar argument also recovers the singular localization of Backelin--Kremnizer.

  \begin{re} In the case of regular $\chi$, one can ask how this treatment of localization compares to earlier ones. The construction of the functors is somewhat different, especially so in affine type. As for the proof, the reduction to an equivalence on Borel-equivariant objects first appeared in work of E. Frenkel--Gaitsgory \cite{fgl}. From there, they rely on an explicit nontrivial calculation by Kashiwara--Tanisaki, while we note it follows directly from the compatibility with the Hecke action. 
  \end{re}

  \begin{re} One can also ask about the strict logical dependence of the arguments in the present paper on prior results.\footnote{Of course, the non-logical dependence on earlier work is immense, and difficult to overstate.} Aside from facts in categorical representation theory (of the flavor of abstract nonsense), our proofs of the main Theorems \ref{t:hch=hecke} and \ref{t:locintro} are self-contained except for some basic facts about the Schubert decomposition of flag varieties and Category $\mathscr{O}$ for $\fsl_2$.

  After this, we do use two nontrivial results from representation theory. To obtain the block decompositions for our categories of $\gk$-modules and bimodules, we use Kac--Kazhdan's calculation of the Shapovalov determinant \cite{kk}. In addition, to obtain the linkage principle for positive energy representations predicted by Yakimov, we need to identify our construction of $\fg\mod_\chi$ with the usual description via infinitesimal characters. This is deduced from Beilinson--Bernstein's calculation of global twisted differential operators on the finite flag manifold \cite{beilinson-bernstein}.

  \end{re}

 \subsubsection{} We would like to finish with a couple comments which lie outside the scope of the article.
 
 First, the presented arguments adapt, {\em mutatis mutandis}, to the case of Kac--Moody groups and blocks containing a simple Verma module with finite Weyl group stabilizer (or the dual of such a block). It would be very interesting to find a similar approach to localization in positive characteristic and for quantum groups. 
 
 Second, while we address the question raised by I. Frenkel--Malikov in \cite{frenkelmalikov1} and \cite{frenkelmalikov2} of finding the desired category of affine Harish-Chandra bimodules, we do not address some striking conjectures of {\em loc. cit.} on the relation with Kazhdan--Lusztig fusion. There indeed is a monoidal functor of the anticipated form from the Kazhdan--Lusztig category to the present category of bimodules. This, along with the closely related theory of projective functors in affine type and principal series Harish-Chandra bimodules, will be discussed in a sequel to the present work.

Third, we would like to mention that forthcoming work of Losev produces categories of Harish-Chandra bimodules for quantum groups. We expect that one has equivalences between these and the categories considered in this paper, compatible with the Kazhdan--Lusztig equivalence.



\subsection{Organization of the paper} Throughout the first half of the paper, we alternate between generalities in categorical representation theory, and their applications to the problem at hand. 

In Section 2, we gather preliminary material and set notation. 

In Section 3, we study the basic adjunctions between representations of Hecke categories and groups. This is then used in Section 4, where we introduce the categories of Lie algebra representations studied in this paper. 

In Section 5, we prove the localization theorems in the maximally singular case. We then study the functor of parabolic induction for categorical representations, and crucially its relation with induction for Hecke categories, in Section 6. This is used in Section 7 to deduce the general case of the localization theorems. 

The remaining sections then apply the previous material to various problems in representation theory, and may be read independently of one another. In Section 8, we obtain a linkage principle for positive energy representations of affine Lie algebras suggested by Yakimov. In Section 9, we introduce equivariant Hochschild cohomology and the analogues of strict infinitesimal characters. Finally, in Section 10, we obtain the main theorems on affine Harish-Chandra bimodules.

\noindent {\bf Acknowledgements.} This paper arises from pairing methods from categorical representation theory with perspectives on Hecke categories and their modules we learned from the work of Soergel and his school. 

For the former subject, we are very grateful to Dennis Gaitsgory, Sam Raskin, and David Yang for stimulating conversations. In particular, we thank Sam Raskin for sharing with us his approach to Kashiwara--Tanisaki localization via convolution with a Verma module, which greatly influenced our thinking. 

For the latter subject, we thank Roman Bezrukavnikov, Ben Elias, Ivan Losev, Shotaro Makisumi, Ben Webster,  Geordie Williamson, and Zhiwei Yun for many helpful conversations. 

We would also like to thank Levent Alpoge, Tony Feng, Sasha Goncharov, David Jordan, Apoorva Khare, Fedor Malikov, Wolfgang Soergel, and Milen Yakimov for stimulating conversations and correspondence related to the present work. 

Finally, particular thanks are due to Igor Frenkel, who shared with us some of his beautiful expectations about affine Harish-Chandra bimodules and principal series for loop groups several years ago.

\section{Preliminaries and notation}
\label{s:pn}

In this section we collect some background material and set notation. The reader may wish to, after possibly glancing at
Sections \ref{sss:grpslies} and \ref{sss:gmodintlevels} for what we denote by $G$, $B$, and $\fg\mod$, move directly to Section \ref{s:hckinv} and refer back only as needed. 

\subsection{Differential graded categories} A reference for the following material is Chapter 1 of \cite{gaitsroz}. 

\subsubsection{}

Throughout the paper, we work over a fixed algebraically closed field $k$ of characteristic zero. We denote by $\DGCat$ the $(\infty, 2)$-category of cocomplete, presentable $k$-linear dg-categories. Unless otherwise specified, by a category $\sC$ we mean an object of $\DGCat$. By a functor between categories $\sC \rightarrow \sD$ we mean an object of the corresponding category\footnote{We view it as a dg-category, using the fact that $\DGCat$ is enriched over itself. It is convenient to express this dg-category as an internal Hom using the Lurie tensor product discussed below. In doing so, one can avoid the theory of $(\infty,2)$-categories entirely if desired.} of maps
$
     \on{Hom}_{\DGCat}(\sC, \sD),
$
i.e. a continuous quasi-functor; see for example \cite{drquo} for a review of quasi-functors.

Throughout the paper, when we consider a limit or colimit of categories, we mean the (co)limit taken in $\DGCat$.

\subsubsection{} We recall that $\DGCat$ carries the structure of a unital symmetric monoidal $\infty$-category with respect to the Lurie tensor product. The unit is given by $\on{Vect}$, the category of $k$-vector spaces, and
we denote the underlying binary product of categories $\sC$ and $\sD$ by $\sC \otimes \sD$. Given objects $c$ of $\sC$ and $d$ of $\sD$, we denote the corresponding object of $\sC \otimes \sD$ by $c \boxtimes d$. 

 If a category $\sC$ is dualizable with respect to the Lurie tensor product, we denote its dual by $\sC^\vee$. Plainly, this is equivalent to, for any category $\sS$, the natural map
 \begin{align*}
      \Hom_{\DGCat}(\sC, \Vect) \otimes \sS &\simeq \Hom_{\DGCat}(\sC, \Vect) \otimes \Hom_{\DGCat}(\Vect, \sS) \\ &\rightarrow \Hom_{\DGCat}(\sC, \sS)
 \end{align*}
 being an equivalence (or in fact, just for $\sS = \sC$), in which case $\sC^\vee \simeq \Hom_{\DGCat}(\sC, \Vect)$. 
 
 Of particular importance to us is the resulting identification of the category of endo-functors of a dualizable category $\scc$
 \begin{equation} \label{e:dualend}
    \scc^\vee \otimes \scc \simeq \on{Hom}_{\DGCat}(\scc, \scc).
 \end{equation}
This is an equivalence of monoidal $\infty$-categories, where on the left the underlying binary product is induced by contraction of tensors using the evaluation pairing
\[
    \on{ev}: \scc \otimes \scc^\vee \rightarrow \on{Vect}.
\]
I.e., it is given by the composition 
\[
   \sC^\vee \otimes \sC \otimes \sC^\vee \otimes \sC \xrightarrow{\on{id}_{\sC^\vee} \otimes \on{ev} \otimes \on{id}_{\sC}} \sC^\vee \otimes \Vect \otimes \hspace{.5mm}  \sC \simeq \sC^\vee \otimes \sC. 
\]

Finally, if $\sC$ and $\sD$ are dualizable, one then has a canonical equivalence 
\[
\on{Hom}_{\DGCat}(\sC, \sD) \simeq \Hom_{\DGCat}(\sD^\vee, \sC^\vee),
\]
and in particular for a functor $F: \sC \rightarrow \sD$
one has a dual functor $F^\vee: \sD^\vee \rightarrow \sC^\vee$. 
\subsection{Groups and Lie algebras}
\label{ss:grpslies}
\subsubsection{}\label{sss:grpslies} Throughout the paper, we use the letter $G$ to denote either (i) a connected reductive group over $k$, or (ii) the loop group of an almost simple, simply connected  group $H$.\footnote{One may deduce from case (ii) similar results, {\em mutatis mutandis}, for (the neutral component of) the loop group of a connected reductive group.} We recall that the latter is an ind-scheme of infinite type over $k$. 

We use the letter $B$ to denote in case (i) a Borel subgroup of $G$, or in case (ii) an Iwahori subgroup. By the latter, we mean the preimage in the arc group of $H$ of a Borel subgroup of $H$. We denote the prounipotent radical of $B$ by $N$, and write $T$ for the abstract Cartan $B/N$. We denote the character and cocharacter lattices of $T$ by $\Lambda$ and $\check{\Lambda}$, respectively, in case (i),  and $\Lambda_f$ and $\check{\Lambda}_f$, respectively, in case (ii). The subscript `$f$' is for `finite'; we will explain what we denote by $\Lambda$ in case (ii) below in Section \ref{ss:liealgreps}. 

\subsubsection{} We denote by $\fg$ the Lie algebra of $G$ in case (i), and the topological Lie algebra in case (ii). Explicitly, if in the latter case we write $\fh$ for the Lie algebra of $H$, we have 
\[
   \fg \simeq \fh \underset k \otimes k(\!(z)\!). 
\]
\subsubsection{} 
Given a level $\kappa$ for $\fh$, i.e. an $\on{Ad}$-invariant symmetric bilinear form on $\fh$, we denote the corresponding Kac--Moody extension by 
\begin{equation} \label{e:centext}
    0 \rightarrow k \cdot \mathbf{1} \rightarrow \widehat{\fh}_\kappa \rightarrow \fh \underset k \otimes k(\!(z)\!) \rightarrow 0. 
\end{equation}
 Recall there exists a {\em basic form} $\kappa_b$, which gives the short coroots squared length two. We call a form $\kappa$ {\em integral} if it is an integral multiple of $\kappa_b$.

Recall that the critical level $\kappa_c$ is by definition minus one half times the Killing form. We say a noncritical level $\kappa$ is {\em positive} if 
\[
        \kappa - \kappa_c \hspace{.5mm} \in \hspace{.5mm} \mathbb{Q}^{> 0} \cdot \kappa_b, 
\]
and otherwise call it {\em negative}.

\subsection{D-modules} Some references for the following material are \cite{rdm} and the very readable \cite{beraldo}. 

\subsubsection{} \label{ss:biggroups} For a quasi-compact, quasi-separated scheme $X$ of finite type over $k$, we denote by $D(X)$ the canonical dg-enhancement of its unbounded derived category of D-modules. 

We will also need categories of D-modules on certain infinite dimensional varieties, which we again denote by $D(X)$; this is denoted by $D^*(X)$ in the above references. 

 For an ind-scheme of ind-finite type, e.g. the affine flag variety, given a presentation as a filtered colimit of varieties of finite type along closed embeddings
$
     X = \varinjlim_\alpha X_\alpha,
$
its category of D-modules is the colimit (in $\DGCat$)
\[
   D(X) \simeq \varinjlim_\alpha D(X_\alpha).
\]
Plainly, $D(X)$ is compactly generated by $*$-pushforwards of coherent complexes of D-modules from all $X_\alpha$, with the natural chain complexes of homomorphisms between such pushforwards. 

We note that $D(X)$, like many categories met in this paper, carries a natural $t$-structure but is not in general (1) the derived category of its heart, nor (2) left complete, i.e. it may contain objects in degrees $\leqslant n$ for all integers $n$ with respect to the $t$-structure.

%
%
%
%
Recall that a scheme $Y$ is called \emph{placid} if it admits a presentation 
$
     Y = \varprojlim_\beta Y_\beta
$
as a cofiltered limit of varieties under affine smooth maps (e.g. the Iwahori subgroup). Its category of D-modules is the colimit 
\[
   D(Y) \simeq \varinjlim_\beta D(Y_\beta).
\]
Plainly, $D(Y)$ is compactly generated by $!$-pullbacks of coherent complexes from all $Y_\beta$, with homomorphisms the colimit of the complexes of homomorphisms of the pullbacks to $Y_{\beta'}$, for all sufficiently large $\beta'$.

Finally, a placid ind-scheme, e.g. the loop group, is an ind-scheme which admits a presentation as a filtered colimit of placid schemes along closed embeddings of finite presentation $Z = \varinjlim_\gamma Z_\gamma$. Its category of D-modules is again the colimit of $*$-pushforwards
\[
     D(Z) \simeq \varinjlim_\gamma D(Z_\gamma).
\]
Given a map $f: X \rightarrow Y$ of placid ind-schemes, one has an associated $*$-pushforward
\[
 f_*: D(X) \rightarrow D(Y). 
\]

\subsubsection{}\label{ss:dualsdmods} Finally, we should note such categories of D-modules are typically self-dual as dg-categories. For a variety $X$, the duality pairing $D(X) \otimes D(X) \rightarrow \on{Vect}$ is given by
\[
  M \boxtimes N \mapsto \Gamma_{\on{dR}}(X, M \overset ! \otimes N). 
\]
For a map of varieties $X \rightarrow Y$, the functors of $*$-pushforward and $!$-pullback are then dual. 

Similarly, for a placid scheme, or ind-scheme of ind-finite type, its category of D-modules is canonically self-dual. For a placid ind-scheme, a choice of {\em dimension theory}, if it exists, yields a self-duality, cf. \cite{rdm}.

\subsection{Categorical representations} Some references for this material are \cite{beraldo}, \cite{mys}, and \cite{whitlocglob}.

\subsubsection{}\label{ss:catreps} If $K$ is a group placid ind-scheme, its category of D-modules $D(K)$ is naturally an algebra object in $\DGCat$, whose underlying binary product is given by convolution
\[
     D(K) \otimes D(K) \simeq D(K \times K) \xrightarrow{ \on{mult}_*} D(K). 
\]
One then has its associated $(\infty, 2)$-category of left modules $D(K)\mod$. In particular, an object of $D(K)\mod$ is a category $\sC$ equipped with an action 
$
     D(K) \otimes \sC \rightarrow \sC
$
which is associative and unital up to coherent homotopy. 

Inversion on $K$ yields a canonical equivalence between $D(K)$ and its opposite algebra. This induces an identification of its categories of left and right modules. In particular, for any $K$-modules $\sC$ and $\sD$ the category of maps $\Hom_{\DGCat}(\sC, \sD)$ is naturally a $K\times K$-module.

As with modules for any algebra object $\sM$ in $\DGCat$, one has a conservative `forgetful' functor
\[
\on{Oblv}: \Hom_{D(K)\mod}(\sC, \sD) \rightarrow \Hom_{\DGCat}(\sC, \sD). 
\]
So, a $D(K)$-equivariant functor $F: \sC \rightarrow \sD$ is a functor between the underlying categories equipped with a datum of $D(K)$-equivariance. 

  We would like to record one useful property of $D(K)$ which is not enjoyed by all algebras. Namely, any datum of lax or oplax equivariance on a functor $F: \sC \rightarrow \sD$ of $D(K)$-modules is automatically strict; this is Lemma D.4.4 of \cite{whitlocglob}. In particular, any adjoint of a $D(K$)-equivariant functor acquires a canonical datum of $D(K)$-equivariance.

\subsubsection{}\label{s:inv1} An important class of categorical representations arises as follows. If $K$ acts on a placid ind-scheme $X$, this yields an action of $D(K)$ on $D(X)$. Moreover, the duality pairing of Section \ref{ss:dualsdmods} carries a canonical datum of $D(K)$-equivariance. 

In particular, taking $X$ to be a point, we obtain the {\em trivial} representation 
\[
  \on{Vect} \in D(K \times K)\mod.\footnote{Just as in the usual representation theory on vector spaces, it is convenient to view the trivial representation as a bimodule, so that the (co)invariants of a representation are again a representation, cf. Equation \eqref{e:defcoinv} below.} 
\]
For any $D(K)$-module $\sC$, we then have its associated invariants and coinvariants, respectively given by 
\begin{equation} \label{e:defcoinv}
      \sC^K := \Hom_{D(K)\mod}(\Vect, \sC) \quad \text{and} \quad \sC_K :=  \Vect \underset {D(K)} \otimes \sC.
\end{equation}
Explicitly, via the bar resolution, one has the semisimplicial presentation  
\begin{equation} \label{e:diag}
\xymatrix{ \sC^{K} \simeq \varprojlim \big( \hspace{.5mm}  \hspace{.5mm} \sC
  \ar@<-.4ex>[r] \ar@<.4ex>[r] &  \hspace{.5mm} \Hom_{\DGCat}(D(K), \sC) \ar@<-.8ex>[r] \ar@<.8ex>[r] \ar[r] & \hspace{.5mm} \Hom_{\DGCat}(D(K) \otimes D(K), \sC)  \hspace{.5mm} \cdots \hspace{.5mm} \big) \hspace{.5mm},   
}
\end{equation}
where the arrows are induced as usual by action and projection; one also has a similar semicosimplicial presentation of the coinvariants. There are tautological `forgetting' and `inserting' $D(K)$-equivariant functors 
\[
   \on{Oblv}: \sC^K \rightarrow \sC \quad \text{and} \quad \on{ins}: \sC \rightarrow \sC_K. 
\]
For $D(K)$-modules $\sC$ and $\sD$, passing from an equivariant functor to the underlying functor between categories yields an equivalence 
\[
     \on{Hom}_{D(K)\mod}(\sC, \sD) \simeq \Hom_{\DGCat}(\sC, \sD)^K,
\]
where on the right-hand side $K$ acts diagonally.

\subsubsection{} \label{s:inv2}If $K$ is an affine group scheme whose prounipotent radical is of finite codimension, then the tautological forgetful functor $\on{Oblv}$ from invariants admits a right `averaging' adjoint
\[
  \on{Oblv}: \sC^K \rightleftarrows \sC: \on{Av}_{*}^K.
\]
Moreover, by adjunction $\on{Av}_{*}^K$ factors as a map 
\[
  \sC \xrightarrow{\on{ins}} \sC_K \rightarrow \sC^K,
\]
and the latter is an equivalence. That is, for `compact open' groups $K$, their invariants and coinvariants are canonically isomorphic, and in particular commute with (co)limits and duality. 

\begin{ex} Suppose for simplicity that $K$ is of finite type, and acts on a variety $X$. Then, by smooth descent for D-modules, one has a canonical equivalence 
\[
   D(X)^K \simeq D(X/K),
\]
where $D(X/K)$ denotes the category of D-modules on the stack quotient $X/K$, cf. \cite{bernlunts}, \cite{drga} for descriptions of the latter. Writing $\pi: X \rightarrow X/K$ for the projection, this exchanges $\on{Oblv}$ with $\pi^*$, and $\on{Av}_*^K$ with $\pi_*$. 
\end{ex}

\subsubsection{} \label{ss:twists}We will also need to consider invariants twisted by a character. So, suppose $\chi$ is a character sheaf on $K$, i.e. a D-module invertible with respect to the $!$-tensor product equipped with an isomorphism
\[
   \on{mult}^! \chi \simeq \chi \boxtimes \chi
\]
which is associative. One has a corresponding $D(K \times K)$-bimodule structure on $\Vect$, which we denote by $\on{Vect}_\chi$, whose associated binary product is given by
\[
    D(K \times K) \otimes \on{Vect}_\chi \simeq D(K \times K) \xrightarrow{\on{mult}_*} D(K) \xrightarrow{ - \overset ! \otimes \chi} D(K) \xrightarrow{\Gamma_{\on{dR}}} \on{Vect}_\chi.
\]
For a $D(K)$-module $\sC$, one then has the twisted invariants  $\sC^{K, \chi}$ and coinvariants $\sC_{K, \chi}$. The results of Sections \ref{s:inv1} and \ref{s:inv2} carry over {\em mutatis mutandis}. 

The only character sheaves we need in this paper arise as follows. On $\mathbb{G}_a$, the exponential D-module exp($z$), concentrated in cohomological degree $-1$, is a character sheaf with respect to a unique structural isomorphism. Given a homomorphism $$\psi: K \rightarrow \mathbb{G}_a, $$we denote by the same letter $\psi$ the character sheaf $\psi^! \on{exp}(z).$ 

Finally, we need to recall the notion of a {generic} additive character of a maximal unipotent group $N$ in a connected reductive group $G$. Namely, the adjoint action of the associated Borel $B$  on $N$ induces an action of the abstract Cartan $T$ on the group scheme of homomorphisms
\begin{equation} \label{e:grphoms}
   \on{Hom}(N, \mathbb{G}_a).
\end{equation}
One says that an additive character $\psi: N \rightarrow \mathbb{G}_a$, viewed as a $k$-point of \eqref{e:grphoms}, is {\em generic} if it lies in the open $T$-orbit. Plainly, if one fixes a pinning of $G$, this identifies \eqref{e:grphoms} with a product of $\mathbb{G}_a$'s indexed by the Dynkin diagram of $G$; a character is generic if and only if its coordinates are all nonzero.

\subsection{Lie algebra representations} Some references for this material are \cite{fg}, \cite{ag}, and  \cite{mys}. 
\label{ss:liealgreps}
\subsubsection{} Fix an affine algebraic group $K$ of finite type. We write $\fk$ for the Lie algebra of $K$, and denote the category of $\fk$-modules by $\fk\mod$. Explicitly, this is the canonical dg-enhancement of the unbounded derived category of its abelian category of representations. 

This carries a canonical structure of $D(K)$-module, which may be seen as follows. Consider $D(K)$ as a $D(K)$-module by right 
convolution, and 
consider the category of weakly equivariant objects
$
D(K)^{K, w}.
$
Taking invariant 
sections of 
the underlying quasicoherent sheaf
yields an equivalence
\[
\Gamma^K: D(K)^{K, w} \simeq \fk\mod.
\]
Since this is important for us, we sketch the argument for the reader's convenience. Note that the left hand side is compactly generated by the algebra of differential operators, with its standard weakly equivariant structure. This is sent via $\Gamma^K$ to the enveloping algebra, and identifies their endomorphisms, which yields the assertion.

In particular, $\fk\mod$ acquires a $D(K)$-action via left convolution. Concretely, the associated action map is of the form 
\begin{equation} \label{e:convgmod}
      D(K) \otimes \fk\mod \rightarrow \fk\mod, \quad M \boxtimes N \mapsto \Gamma(K, M) \underset{\fk} \otimes N,
\end{equation}
where the appearing copy of $\fk$ under the tensor product is the $K \times K$ invariant vertical vector fields with respect to the multiplication map $K \times K \rightarrow K$.

Moreover, by the canonical identification of weak invariants and coinvariants, one has for any $D(K)$-module $\sC$ the identity
\begin{equation}
    \label{e:gmodweakinvs} \Hom_{D(K)\mod}(\fk\mod, \sC) \simeq \sC^{K, w}.
\end{equation}

By restriction, for any subgroup $P \subset K$ the category $\fk\mod$ is a $D(P)$-module. The category of equivariant objects $\fk\mod^{P}$ is canonically equivalent to the derived category of modules for the Harish-Chandra pair $(\fk, P)$, cf. Appendix A of \cite{whit}, Section 5 of \cite{tfle}, or Section 9 of \cite{mys}. In particular, if $K$ is a connected reductive group and $P$ is a Borel subgroup, we obtain a version of Category $\mathscr{O}$, or more precisely the part supported on the weight lattice.

\subsubsection{} \label{sss:infdimliealgreps}Let us turn to the infinite-dimensional case. Suppose $K$ is a group placid ind-scheme which is formally smooth and  {Tate}. Recall that the latter means that $K$ admits a closed embedding $K_\circ \hookrightarrow K$ where $K_\circ$ is an affine group subscheme of $K$ and $K/K_\circ$ is an ind-scheme of ind-finite type. We refer to such a subgroup $K_\circ$ as a compact open subgroup in what follows. 

Under this assumption, the topological Lie algebra $\fk$ of $K$ is a Tate Lie algebra. For a fixed continuous central extension $$0 \rightarrow k \cdot \mathbf{1} \rightarrow \fkk \rightarrow \fk \rightarrow 0,$$its renormalized category of representations is defined as follows. Consider the abelian category 
$
    \fkk\mod^\heartsuit
$
of smooth representations of $\fkk$ on which $\mathbf{1}$ acts by the identity. Within its bounded derived category $\fkk\mod^b$, consider the pretriangulated envelope $\sS$ of all representations induced from finite dimensional smooth representations of compact open subalgebras. By definition, $\fkk\mod$ is the ind-completion of $\sS$. 

Recall that $K$ admits a canonical extension by the multiplicative group, namely its Tate extension$$1 \rightarrow \mathbb{G}_m \rightarrow \widehat{K} \rightarrow K \rightarrow 1.$$Suppose that $\kappa$ is a $k$-multiple of the associated Tate extension of $\fk$. In this case, one has a monoidal category $D_\kappa(K)$ of twisted D-modules, and 
$\fkk\mod$ is naturally a $D_\kappa(K)$-module. Concretely, the action is similar to \eqref{e:convgmod}, with Lie algebra homology replaced by semi-infinite homology, cf. Section 22 of \cite{fg2}.

 For any compact open subgroup $K_\circ \subset K$, the twisting $\kappa$ is canonically trivial upon restriction to $K_\circ$, and the category of equivariant objects 
 \begin{equation} \label{e:hchmod}
    \fkk\mod^{K_\circ}
 \end{equation}
may be described as follows. Within the bounded derived category of modules for the Harish-Chandra pair $(\fkk, K_\circ)$, consider the pretriangulated envelope $\sS$ of modules induced from finite dimensional $K_\circ$-modules. Then \eqref{e:hchmod} is canonically equivalent to the ind-completion of $\sS$. In particular, for $K$ the loop group of an almost simple, simply connected group, and $K_\circ$ the Iwahori subgroup, we obtain a version of affine Category $\mathscr{O}$ at level $\kappa$, or more precisely the part supported on the weight lattice $\Lambda_f$. 

\subsubsection{} \label{sss:gmodintlevels}
We now discuss two pieces of nonstandard notation which will be convenient in what follows. Let $G$ and $\hk$ be as in case (ii) of Section \ref{ss:grpslies}. Note that if $\kappa$ is integral, then $\hk\mod$ carries an action of $D(G)$, i.e. one may trivialize the twisting. We may then define the $D(G)$-module 
\begin{equation} \label{e:defgmod}
  \fg\mod := \underset \kappa \oplus \hspace{1mm} \hk\mod,
\end{equation}
where $\kappa$ runs over the set $\mathscr{K}$ of noncritical integral levels, i.e. 
$
      \mathscr{K} = (\mathbb{Z} \cdot \kappa_b) \setminus \kappa_c.
$
Similarly, let us write $\Lambda$ for the set of highest weights for modules in $\fg\mod^B$. Plainly, we have 
\[
     \Lambda \simeq \Lambda_f \times \mathscr{K},
\]
where the pair $(\lambda_f, \kappa)$ corresponds to the Verma module for $\hk$ of highest weight $\lambda_f$. Recall that at an integral level $\kappa$, the dot action of the affine Weyl group $W$ preserves $\Lambda_f$. In particular, it makes sense to speak of the dot action of $W$ on $\Lambda$. 

As a final piece of notation related to our definition of $\Lambda$, let us write $\rho_f \in \Lambda_f$ for the half sum of the positive roots of $\fh$, and $\rho$ for the element $(\rho_f, \kappa_c)$ of $\Lambda_f \times \mathbb{Z} \cdot \kappa_b$. With this, we have the involution 
\[
     \lambda \mapsto -\lambda - 2 \rho: \Lambda \rightarrow \Lambda,
\]
which plays a basic role in cohomological duality. Explicitly, this sends a pair $(\lambda_f, \kappa)$ to $(-\lambda_f - 2 \rho_f, -\kappa + 2 \kappa_c)$.

\subsubsection{} \label{sss:eqindres}It will be important for us that the usual functors of induction and restriction of Lie algebra representations are maps of categorical group representations. Since we are unsure if this already appears in the literature, we include proofs. 

Suppose $K_1, K_2$ are a pair of groups as in Section \ref{sss:infdimliealgreps}, with corresponding Tate Lie algebras $\fk_1, \fk_2$. Suppose $h: K_1 \rightarrow K_2$ is a homomorphism of group ind-schemes and $\fk_{2, \kappa}$ is a central extension  of $\fk_2$ such that it and its pullback $\fk_{1, \kappa}$ to $\fk_1$ are $k$-multiples of the Tate extension. 

\begin{lem} The associated restriction functor 
\[
  \on{Res}:  \fk_{2, \kappa}\mod \rightarrow \fk_{1, \kappa}\mod
\]
carries a canonical datum of $D_\kappa(K_1)$-equivariance.
\end{lem}

\begin{proof} Following the notation of Section 9 of \cite{mys}, restriction identifies with the composition 
\[
   \fk_{2, \kappa}\mod \simeq D^!_\kappa(K_2)^{K_2, w} \xrightarrow{\on{Oblv}} D^!_\kappa(K_2)^{K_1, w} \xrightarrow{h^!} D^!_\kappa(K_1)^{K_1, w} \simeq \fk_{1, \kappa}\mod,
\]
and the two intermediate functors carry canonical data of $D_\kappa(K_1)$-equivariance.    \end{proof}

By passing to left adjoints, we obtain the following. 

\begin{cor} If $K_1$ is a compact open subgroup of $K_2$, the functor of induction 
\[
\on{ind}: \fk\mod \rightarrow \fk_{2, \kappa}\mod
\]
carries a canonical datum of $D(K_1)$-equivariance. 
\end{cor}

\subsubsection{} \label{ss:liealgrepsd}As for D-modules, the appearing categories of Lie algebra representations are canonically self-dual. Namely, for any finite dimensional Lie algebra $\mathfrak{k}$, the functor of Lie algebra homology
\[
   \mathfrak{k}\mod \otimes \hspace{.5mm} \mathfrak{k}\mod \rightarrow \Vect, \quad \quad M \boxtimes N \mapsto M \underset{\mathfrak{k}}\otimes N
\]
is a perfect pairing. In particular, one has a canonical equivalence of monoidal $\infty$-categories
\[
   \mathfrak{k} \oplus \mathfrak{k}\mod \simeq \mathfrak{k}\mod \otimes \hspace{.5mm} \mathfrak{k}\mod \simeq \on{Hom}_{\DGCat}( \mathfrak{k}\mod, \mathfrak{k}\mod),
\]
cf. \eqref{e:dualend}. Plainly, this exchanges the tensor product of bimodules with the composition of endofunctors.

Let us turn to the affine Lie algebra $\hk$, cf. Section \ref{ss:grpslies}. Consider the functor of semi-infinite homology, which we normalize with respect to the pro-Lie algebra of its Iwahori subgroup,
\[
    \hk\mod \otimes \hspace{.5mm} \widehat{\fh}_{-\kappa  + 2\kappa_c}\mod \rightarrow \Vect, \quad M \boxtimes N \mapsto M \underset{\widehat{\fh}}{\overset{\frac{\infty}{2}}\otimes} \hspace{.5mm} N.
\]
It was shown by Arkhipov--Gaitsgory that this is a perfect pairing \cite{ag}; for an alternative proof cf. Section 5 of \cite{lpw}. By Section 9 of \cite{mys} this functor carries a canonical datum of $D_\kappa(G)$-equivariance, i.e. the duality is one of categorical representations.  

\begin{re}We remind that the Lie algebra of the Iwahori only appears to normalize the functor; any other compact open subalgebra would do, and the resulting functor would differ by tensoring with a cohomologically graded line. \end{re}

One again has a canonical equivalence of monoidal $\infty$-categories
\[
    \hk \oplus \widehat{\fh}_{-\kappa + 2 \kappa_c}\mod \simeq \hk\mod \otimes \hspace{.5mm} \widehat{\fh}_{-\kappa + 2 \kappa_c}\mod \simeq \on{Hom}_{\DGCat}(\hk\mod, \hk\mod);
\]
this exchanges the semi-infinite tensoring of bimodules with the composition of endo-functors. 

By summing over the noncritical integral levels $\mathscr{K}$, and recalling that $\kappa_c$ is integral, it follows that $\fg\mod$ is self-dual, and that one has a canonical equivalence $$\fg\mod \otimes \hspace{.5mm}\fg\mod \simeq \on{Hom}_{\DGCat}(\fg\mod, \fg\mod).$$
Concretely, to compose bimodules 
\[
   M \in \widehat{\fh}_{\kappa_1} \oplus \widehat{\fh}_{\kappa_2}\mod, \quad \quad N \in \widehat{\fh}_{\kappa_3} \oplus \widehat{\fh}_{\kappa_4}\mod,
\]
one tensors them if the anomaly vanishes, and otherwise sets their composition to zero, i.e. 
\[
    M \star N := \begin{cases}M \underset{\widehat{\fh}}{\overset{\frac{\infty}{2}}\otimes} \hspace{.5mm} N \in \widehat{\fh}_{\kappa_1} \oplus \widehat{\fh}_{\kappa_4}\mod & \on{if } \kappa_2 + \kappa_3 = 2\kappa_c, \\ 0 & \on{otherwise.} \end{cases}
\]

\section{Hecke algebras and invariant vectors}\label{s:hckinv}

In this section we record, in the present categorical context, the fundamental adjunctions between representations of a group and its Hecke 
algebras, as well as their basic properties.

The statements and proofs in this section have the flavor of abstract nonsense. The reader mostly interested in representation-theoretic applications may wish to glance at the statements of Theorem \ref{t:cuff}, Definition \ref{d:genkinvs}, and Corollary \ref{c:checkonhwvs} before proceeding directly to Section \ref{s:catliereps}.

 As with later sections containing abstract nonsense, many of the basic results are known to experts.

\subsubsection{}   We first
recall the analogous theory 
for actions of groups on vector spaces; the reader may wish to skip this and proceed directly to Section \ref{ss:adjs}. 

  Suppose that $H$ is a finite group, and write $H\on{-mod}$ for its category of representations 
over a 
field of characteristic zero. Given a subgroup $K$ of $H$, one can consider the associated functor 
of $K$-invariants. This is co-represented by the space of 
functions $\on{Fun}(H/K)$, whose endomorphisms are the convolution algebra $\on{Fun}(K \backslash H / K).$
In particular, we obtain an adjunction 
\begin{equation} \label{e:fgadj}
    \on{Fun}(H/K) \underset{\on{Fun}(K \backslash H / K)}{\otimes} -  
    :\on{Fun}(K \backslash H/ K)\on{-mod} \rightleftarrows 
        H\on{-mod}: \on{inv}^K. 
\end{equation}
The left adjoint in \eqref{e:fgadj} is fully faithful, with essential 
image generated by the simple representations which 
have nontrivial $K$ invariants. In particular, for a representation $\pi$ of 
$H$, the co-unit 
\begin{equation} \label{e:fgcou}
   \on{Fun}(H/K) \hspace{.5cm} \underset{\hspace{-.5cm}\on{Fun}(K \backslash H 
   / 
   K)}{\otimes} \pi^K 
   \rightarrow \pi
\end{equation}
is an embedding, with essential image the corresponding isotypic components of 
$\pi$, i.e. the submodule generated by $K$ invariants.

We now turn to the analogous results for categorical representations. 

\begin{re} For a $p$-adic group $H$ with a compact open subgroup $K$, an adjunction similar to \eqref{e:fgadj} plays a fundamental role in its representation theory. While this is a more direct analogue of our study of affine Hecke categories and loop groups, the categorical theory behaves more similarly to the case of a finite group, including crucially the fully faithfulness of \eqref{e:fgcou}. 
\end{re}

\subsection{The basic adjunctions} \label{ss:adjs}

\subsubsection{} Suppose $K \rightarrow H$ is a homomorphism of group placid ind-schemes, cf. Section \ref{ss:biggroups}. For 
$D(H)$-modules, the functor of $K$ invariants is given by
\[ \sHom_{D(K)\mod}(\on{Vect}, -) \simeq 
\sHom_{D(H)\mod}(D(H)_K, 
-).
\]
In particular, such invariants carry a right action by 
\[
    \sH := \Hom_{D(H)\mod}(D(H)_K, D(H)_K),
\]
and we obtain a tensor-hom adjunction
\begin{equation} \label{e:invs}
   - \underset {\HK} \otimes \hspace{.5mm} D(H)_K: \on{mod-}\hspace{-.5mm}\HK \rightleftarrows D(H)\mod: -^K.
\end{equation}
In this generality, the right adjoint in \eqref{e:invs} need not be $\DGCat$-linear, nor commute with colimits.

%
%
%
%
	%
	%

\subsubsection{} \label{sss:compactopen}

Let us further suppose that $K$ is 
a affine group subscheme of $H$  whose 
prounipotent radical is of finite codimension. In this 
case, one has a canonical identification of invariants and coinvariants with respect to $K$, cf. Section \ref{s:inv2}. In particular, we may canonically identify  $\HK$ with 
the convolution algebra of bi-$K$-equivariant D-modules
\[ \HK \simeq D(K \backslash H / K).
\]
Inversion on $H$ induces a canonical isomorphism between $\sH$ and 
its opposite algebra, and in particular between its categories of 
left and right 
modules. We use both of these identifications implicitly going forwards.

\begin{pro} \label{p:fftensoringup}Under our assumption on $K$, the left 
adjoint in 
\eqref{e:invs} is fully faithful. 
\end{pro}

\begin{proof}
    By its identification with $K$ coinvariants, the functor of $K$ invariants 
    commutes with 
	colimits and is $\DGCat$-linear. It follows, e.g. via 
	using the bar resolution, that for any $\HK$-module $\scc$ the 
	unit 
	identifies with the composition
	\[
	\scc \simeq D(K 
	\backslash H / K) \underset{D(K\backslash H / K)}{\otimes} \scc \simeq \big(D(H/K) \underset{D(K \backslash H / K)}\otimes \sC\big)^K.
	\]
\end{proof}

\subsubsection{} \label{sss:flagvar}We continue to assume that $K$ is as in 
Section 
\ref{sss:compactopen}, and make the further assumption that $H/K$ is 
ind-proper (and in particular, of ind-finite type). This yields the following property of the counit, which will be 
important in what follows.

\begin{theo}\label{t:cuff} With our assumptions on $K$, for any 
$D(H)$-module 
$\sC$, the counit
	\begin{equation} \label{e:ffemb}
	  D(H/K) \underset{D(K \backslash H / K)}{\otimes} \sC^K 
	  \xrightarrow{\hspace{2mm} c \hspace{2mm}} \sC
	\end{equation}
	is fully faithful and admits a (continuous) $D(H)$-equivariant right adjoint 
	$c^R$. 
\end{theo}

Informally, the theorem allows one to reconstruct a piece of the entire module from its $K$ invariants, which are often easier to analyze directly. The result and its proof are similar to work of Ben-Zvi--Gunningham--Orem \cite{BZO}; for a generalization see also the interesting paper of Yang \cite{yang}.

Before giving the proof, we mention an important special case. 

\begin{ex} \label{ex:moninvs}Suppose that $K = H$. In this case we refer to the 
obtained category as the monodromic invariants, i.e.
	\begin{align}
	    \sC^{K\hspace{-.6mm} \on{-mon}} &:=   \sC^K \underset{D(K\backslash K / K)} \otimes D(K/K)
	    \\\label{e:moninv} &\simeq 
	      \sC^K \underset{D(\on{pt}\hspace{-.5mm}/K)} \otimes \on{Vect} .  
	\end{align}
Theorem \ref{t:cuff} identifies this with the full subcategory of $\sC$ generated under colimits by the essential 
image of the forgetful functor 
$
    \on{Oblv}: \sC^K \rightarrow \sC.  
$	
%

So, in this case Theorem \ref{t:cuff} reconstructs a `formal 
neighborhood' of  
$\scc^K$ inside of $\scc$ using the action of $D( \on{pt} \hspace{-.5mm}/ K)$ on 
the former. This may be thought of as a version of de-equivariantization. \end{ex}

\begin{proof}[Proof of Theorem \ref{t:cuff}] The claimed equivariance of such a right adjoint is automatic for $D(H)$-modules, cf. Section \ref{ss:catreps}.

By the $\DGCat$-linearity and continuity of the appearing functors, it suffices to show that the $D(H) 
\otimes 
D(H)$-equivariant 
functor
	\begin{equation}\label{e:univcase} D(H/K) \underset{D(K 
	\backslash H / 
	K)}{\otimes} D(K 
	\backslash H) \rightarrow D(H)
	\end{equation}
	is fully faithful and admits a
	right 
	adjoint. Namely, the general case follows by tensoring over $D(H)$ with $\sC$.

	We will prove the assertions for \eqref{e:univcase} by an analysis of the situation on zero-simplices in the bar resolution, along 
	with 
	a monadicity argument. Consider the 
	convolution
	correspondence \[
	\begin{tikzcd}
	& H \overset{K}{\times} H \arrow{dl}[swap]{p} 
	\arrow{dr}{q} & \\
	H/K \times K \backslash H & & H.
	\end{tikzcd} \]
	Since $p$ is a $K$-torsor, the functor $p_*$ admits a 
	left adjoint $p^*$ 
	(see 
	\cite{rdm}, 
	Proposition 6.18.1). With 
	this, by definition the $D(H) \otimes D(H)$-equivariant 
	composition
	\[
	 D(H/K) \otimes D(K \backslash H) \rightarrow D(H/K) 
	 \underset{D(K 
	 \backslash H / K)}{\otimes} D(K \backslash H) \rightarrow 
	 D(H)
	\]
	 identifies equivariantly with the composition
	\begin{equation}
	m: D(H/K) \otimes D(K \backslash H) 
	\stackrel{p^*}{\longrightarrow} 
	D(H 
	\overset{K}{\times} H) \stackrel{q_*}{\longrightarrow} 
	D(H).
	\label{e:convcompuniv}
	\end{equation}
Since $q$ is 
ind-proper of ind-finite type, the $D(H) \otimes D(H)$-equivariant functor 
$q_*$ admits a right 
adjoint 
$q^!$. By composition, it follows that $m$ admits a right adjoint $m^R$.

	Let us now consider the situation after tensoring over $D(K 
	\backslash 
	H / 
	K)$. We first claim that the insertion 
	\[
	 \msf{ins}: D(H / K) \otimes D(K \backslash H) \rightarrow 
	 D(H / K) 
	 \underset{D(K \backslash H / K)}{\otimes}	D(K \backslash 
	 H)\]
	is monadic. More precisely, we claim it admits a conservative right adjoint 
	$\msf{ins}^R$.

	To see this, consider the presentation of $ D(H / K) 
	\underset{D(K 
		\backslash H / K)}{\otimes} D(K 
	\backslash H) $ via the semisimplicial bar resolution
	%
	\begin{equation} \label{e:barres}
	\begin{tikzcd}
	   \cdots  \arrow[r, shift left = 2] \arrow[r] \arrow[r, shift right = 2] 
	   &  D(H/K) \otimes D(K  \backslash H / K) \otimes 
	   D(K 
	   \backslash 
	  H)\arrow[r, shift left=1] \arrow[r, shift 
	  right=1] &   D(H/K) \otimes D(K 
	  \backslash H).
	 	\end{tikzcd}
	\end{equation}
	It follows from the preceding paragraph that each face map admits a right adjoint. Under the canonical equivalence	of the colimit with the limit over its right adjoints, $\msf{ins}^R$ 
	identifies with evaluation on zero-simplices, which yields the desired 
	claim.

	In term of the bar resolution, \eqref{e:univcase} is given by the 
	augmented semisimplicial 
	dg-category
	\[
	\begin{tikzcd}
	\cdots D(H/K) \otimes D(K 
	\backslash H / K) \otimes D(K \backslash 
	H) \arrow[r, shift left=1] \arrow[r, shift right=1]
	&
	D(H/K) \otimes D(K \backslash H) 
	\arrow[r, "q_*p^*"] & D(H).
	\end{tikzcd}
	\]
    Thus the functor \eqref{e:univcase} 
	amounts to a morphism of monads \[ \msf{ins}^R \msf{ins} \longrightarrow 
	p_*q^!q_*p^*. \]  We will prove 
	that this is an isomorphism, which is equivalent to the fully faithfulness of 
	\eqref{e:univcase}.
	
	Corollary C.2.3 in \cite{1affine} says that the 
	monad $\msf{ins}^R\msf{ins}$ identifies with
	\begin{align*}
	D(H/K) \otimes D(K \backslash H) 
	&\xrightarrow{\id_{D(H/K)} 
	\otimes 
	\msf{act}_{D(K 
			\backslash H)}^R} D(H/K) \otimes D(K \backslash H / K) \otimes 
			D(K 
			\backslash H) \\ 
	&\xrightarrow{\msf{act}_{D(H/K)} \otimes \id_{D(K \backslash 
	H)}} 
	D(H/K) 
	\otimes 
	D(K 
	\backslash H)
	\end{align*}
	as an endofunctor. Thus it is given by pull-push along the correspondence \[
	\begin{tikzcd}[column sep ={3cm,between origins}]
	& & H \overset{K}{\times} H \overset{K}{\times} H \arrow{dl} \arrow{dr} & & 
	\\
	& H/K \times K \backslash H \overset{K}{\times} H \arrow{dl} \arrow{dr} & & 
	H 
	\overset{K}{\times} H/K \times K \backslash H \arrow{dl} \arrow{dr} & \\
	H/K \times K \backslash H & & H/K \times K \backslash H/K \times K 
	\backslash H 
	& & H/K \times K \backslash H.
	\end{tikzcd} \]
	But the endofunctor $ 
	p_*q^!q_*p^*$ 
	is given by pull-push along the isomorphic correspondence \[
	\begin{tikzcd}[column sep ={3cm,between origins}]
	& & H \overset{K}{\times} H \overset{K}{\times} H \arrow{dl} \arrow{dr} & & 
	\\
	& H \overset{K}{\times} H \arrow{dl} \arrow{dr} & & H \overset{K}{\times} H 
	\arrow{dl} \arrow{dr} & \\
	H/K \times K \backslash H & & H & & H/K \times K \backslash H,
	\end{tikzcd} \]
	and a diagram chase shows that the resulting natural isomorphism 
	$\msf{ins}^R\msf{ins} 
	\tilde{\to} p_*q^!q_*p^*$ 
	identifies 
	with the aforementioned morphism of monads. This completes the proof of the 
	fully faithfulness of \eqref{e:univcase}.

	It remains to furnish \eqref{e:univcase} with a right adjoint $\eqref{e:univcase}^R$. 
	To see this, recall that the morphism $m$ identifies with the 
	composition 
	\[ D(H / K) \otimes D(K \backslash H) \xrightarrow{\msf{ins}} 
	D(H / K) 
	\underset{D(K \backslash H / K)}{\otimes}  D(K \backslash H)  
	\xrightarrow{\eqref{e:univcase}} D(H).
	\]
	%
	%
	The desired continuity of $\eqref{e:univcase}^R$, which exists as a discontinuous functor by the adjoint functor theorem, follows via the conservativity 
	and 
	continuity of 
	$\msf{ins}^R$ from the continuity of $m^R$. 
	\end{proof}

\subsubsection{}Let us deduce a useful consequence of Theorem 
\ref{t:cuff}. 
In particular, we keep its assumptions on $K$.

\begin{cor} \label{c:rad} For $K$ as in Theorem \ref{t:cuff}, the functor 
$D(H/K) 
\underset{D(K\backslash H / K)}{\otimes} - $ 
is also right adjoint 
to the functor of $K$ invariants, with unit 
	\[
	  \sC \xrightarrow{ \hspace{1mm} c^R \hspace{1mm}} D(H/K) 
	  \underset{D(K 
	  \backslash H / K)}{\otimes} 
	 \sC^K. 
	\]
\end{cor}
\begin{proof} For a $D(H)$-module $\sC$ and $D(K\backslash H / 
K)$-module 
$\sM$, we have a map 
	\begin{align*}
\Hom_{D(K \backslash H / K)\mod}(\sC^K, \sM) &\overset{\ref{p:fftensoringup}}{\simeq} 
\Hom_{D(H)\mod}\big(D(H/K) 
\underset{D(K\backslash H / K)}{\otimes} \sC^K,D(H/K) 
\underset{D(K\backslash H / K)}{\otimes}  \sM \big) \\ & \hspace{1mm}\rightarrow 
\Hom_{D(H)\mod}\big(\sC, D(H/K) \underset{D(K\backslash H / 
K)}{\otimes}  \sM\big).
	\end{align*}
By Theorem \ref{t:cuff}, the latter arrow is a full embedding. Its 
essential surjectivity follows from the naturality of $c^R$, as it is an 
equivalence when applied to $D(H/K) \underset{D(K\backslash H / 
K)}{\otimes}  
\sM$. 
\end{proof}

\begin{re} \label{c:tensor=hom} Note that Corollary \ref{c:rad} can be rephrased as a natural isomorphism of functors
	\begin{equation}
	\label{e:ideofadjs}
	D(H/K) 
	\underset{D(K\backslash H / K)}{\otimes} (-)   \hspace{2mm} 
	\simeq\hspace{2mm}  \Hom_{D(K \backslash 
	H 
	/ K)\mod}( D(H/K), -).
	\end{equation}
\end{re}

\subsection{Generation}

\subsubsection{} Let $H$ and $K$ be as in Theorem \ref{t:cuff}. In the remainder of this section, we collect some tools which allow one to control a categorical representation from its $K$ invariants. The basic definition is the following.  
\label{s:genbykinvs}

\begin{defn} \label{d:kgen} For a $D(H)$-module $\scc$, we write $\scc^{K\hspace{-.6mm} \on{-gen}}$ for the {\em submodule of $\scc$ 
generated by its 
$K$ invariants}, i.e. the essential image of the counit \eqref{e:ffemb}:
	\[
	     D(H/K) \underset{D(K\backslash H / K)}{\otimes} \sC^K \hspace{1mm}
	     \overset \sim \longrightarrow \hspace{1mm}\scc^{K\hspace{-.6mm} \on{-gen}} \hspace{1mm} \subset \hspace{1mm} \scc. 
	\]
We say that $\scc$ is {\em generated by its 
$K$ invariants} if $\scc^{K\hspace{-.6mm} \on{-gen}} = \scc$, i.e. if the inclusion is an equivalence. 
	\label{d:genkinvs}
\end{defn}
\subsubsection{} Suppose ${\Theta}: \scc \rightarrow \sD$ is a morphism of 
$D(H)$-modules, and consider  the associated 
morphism of $K$-invariant categories $${\Theta}^K: \scc^K \rightarrow \sD^K.$$If ${\Theta}$ is fully faithful, it formally follows that ${\Theta}^K$ is as well. Indeed, given a general monoidal category $\sM$, a fully faithful map of 
$\sM$-modules $\sC 
\rightarrow \sD$, and another $\sM$-module $\sS$, the map 
\begin{equation}
\label{e:ffembs}
   \Hom_{\sM\mod}(\sS, \sC) \rightarrow \Hom_{\sM\mod}(\sS, \sD)
\end{equation}
is a fully faithful embedding.

The following partial converse, which we will make significant use of in what follows, allows one to check fully faithfulness, and in particular equivalences, on 
`highest weight vectors'.

\begin{cor} \label{c:checkonhwvs} If $\scc$ is generated by $K$ invariants and ${\Theta}^K$ is fully faithful, then ${\Theta}$ also is fully faithful.
\end{cor}

\begin{proof} 

By our assumption on $\sC$, we obtain a canonical factorization of $\Theta$ as a composition 
\begin{equation}  \label{e:2stepsff}
  \sC^{K\hspace{-.6mm} \on{-gen}} \rightarrow  
  \sD^{K\hspace{-.6mm} \on{-gen}} \rightarrow \sD.
\end{equation}
By Theorem \ref{t:cuff} the second arrow of \eqref{e:2stepsff} is fully 
faithful, so it is enough to show the first arrow of \eqref{e:2stepsff} is also fully 
faithful. However, as in Remark \ref{c:tensor=hom} we may rewrite this as 
\[
   \Hom_{D(K \backslash H / K)\mod}\big( D(H/K), \sC^K\big) \rightarrow 
   \Hom_{D(K 
   \backslash H / K)\mod}\big( D(H/K), \sD^K\big),
\] 	
hence its fully faithfulness follows from \eqref{e:ffembs}. 	
\end{proof}

 Let us mention two other consequences of Theorem \ref{t:cuff}, namely the favorable interaction of generation by $K$ invariants with passing to subcategories and gluing. Since we do not use these in what follows, we omit their formal deductions.

\begin{cor} If $\scc$ is generated by $K$ invariants, and $\mathring{\scc} 
\rightarrow \scc$ is a fully faithful embedding of $D(H)$-modules, then 
$\mathring{\scc}$ is generated by $K$ invariants. \end{cor}

\begin{cor} Suppose $\sM'$, $\sM$ and $\sM''$ are $D(H)$-modules fitting 
into a 
recollement 
	\[
	     \sM' \rightleftarrows \sM \rightleftarrows \sM''.
	\]
	Then $\sM$ 
	is generated by $K$ invariants if and only if $\sM'$ and $\sM''$ are. 
\end{cor}

\subsubsection{}Finally, let us note the compatibility of $K$ generation and 
duality. Recall that if $\scc$ is a $D(H)$-module which is dualizable as an 
object of $\DGCat$, its dual $\scc^\vee$ carries a canonical structure of 
$D(H)$-module.

\begin{cor} \label{c:dualskgen}For $H$ and $K$ as in Theorem \ref{t:cuff}, and a dualizable 
$D(H)$-module $\scc$, there are canonical isomorphisms 
	\[
	    (\scc^{K\hspace{-.6mm} \on{-gen}})^\vee \simeq (\scc^\vee)^{K\hspace{-.6mm} \on{-gen}}.
	\]
\end{cor}

\begin{proof} Recall the adjunction of Theorem \ref{t:cuff}, i.e. $\scc^{K\hspace{-.6mm} \on{-gen}} 
	\rightleftarrows \scc.$ The fully faithfulness of the left adjoint and the 
	dualizability of $\scc$ straightforwardly imply the dualizability of 
	$\scc^{K\hspace{-.6mm} \on{-gen}}$. By our assumption on $K$, duality interchanges the adjunctions 
	\begin{equation} \label{e:intadjs}
	\on{Oblv}: \scc^K \rightleftarrows \scc: \on{Av}^H_* \quad \quad 
	\text{and} 
	\quad \quad  \on{Oblv}: (\scc^\vee)^K \rightleftarrows \scc^\vee: 
	\on{Av}^H_*.
	\end{equation}	
	It follows that, viewed as full subcategories of $\scc^\vee$, we have 
	\begin{equation} \label{e:inc1}
	     (\scc^\vee)^{K\hspace{-.6mm} \on{-gen}} \subset (\scc^{K\hspace{-.6mm} \on{-gen}})^\vee.
	\end{equation}
	Interchanging the roles of $\scc$ and $\scc^\vee$, we similarly have 
	\begin{equation} \label{e:inc2}
	     \scc^{K\hspace{-.6mm} \on{-gen}} \subset ( (\scc^\vee)^{K\hspace{-.6mm} \on{-gen}})^\vee.
	\end{equation}
	Comparing \eqref{e:inc1} and \eqref{e:inc2}, it follows both inclusions are 
	equivalences, as desired. 
\end{proof}

\section{A category of Lie algebra representations}
\label{s:catliereps}
In this section, we record the definition and first properties of the 
categories of Lie algebra representations considered in this paper. 

\subsubsection{} Recall that $G$ denotes a reductive group or loop group, $B$ denotes a Borel subgroup of $G$, $\fg$ denotes the Lie algebra of $G$, and $\fg\mod$ denotes its category of
representations, cf. Sections \ref{ss:grpslies} and \ref{ss:liealgreps} for 
our precise conventions. 

\subsubsection{} Recall that $\fg\mod$ carries a canonical action of 
$D(G)$, cf. Section \ref{ss:liealgreps}. So, 
as in Definition \ref{d:kgen} we may consider 
	\[
	   \fg\mod^{B\hspace{-.6mm} \on{-gen}} := D(G/B) \underset{D(B \backslash G / 
	   B)}\otimes 
	   \fg\mod^B 
	   \hspace{.2cm} \subset \hspace{.2cm} \fg\mod. 
	\]

We next isolate certain submodules. To do so, we begin with some preliminary 
generalities.  

\subsubsection{} Given a $D(G)$-module $\sC$ and object $\xi$ of $\sC$, 
one 
may 
consider the $D(G)$-submodule generated by $\xi$, i.e. the minimal full 
subcategory of $\sC$ stable under the action of $D(G)$ which contains 
$c$. 
Explicitly, convolution with $\xi$ gives an equivariant functor
\begin{equation} \label{e:convc}
    D(G) \rightarrow  \sC,
\end{equation}
and the desired full subcategory of $\sC$ is the closure of the essential image 
of \eqref{e:convc} under colimits.

\begin{lemma} \label{l:gencomps}Suppose that $K$ is a subgroup of $G$ for which 
$G/K$ is 
ind-proper, and that $\xi$ is equipped with a datum of $K$-equivariance. If we 
write 
$\sS$ for 
the $D(K \backslash G / K)$-submodule of $\sC^K$ generated by $\xi$, 
then 
the 
$D(G)$-submodule generated by $\xi$ is canonically equivalent to 
\begin{equation}
   D(G/K) \underset{D(K\backslash G / K)}{\otimes} \sS.
\end{equation}
\end{lemma}

\begin{proof} This follows from Theorem \ref{t:cuff}. \end{proof}

\subsubsection{}\label{sss:gmodlambda} Recall the set of integral weights $\Lambda$, and that 
$\fg\mod^B$ is compactly generated by the Verma modules $M_\lambda$, for 
$\lambda \in \Lambda$, cf. Sections \ref{ss:grpslies} and \ref{ss:liealgreps}. 

\begin{defn} For $\lambda \in \Lambda$, denote the 
$D(G)$-submodule of $\fg\mod$ generated by $M_\lambda$ by 
\[
 i_!:  \fg\mod_{\lambda} \rightarrow \fg\mod. 
\]
\end{defn}

We next identify compact generators for the $B$ invariants of $\fg\mod_{\lambda}$. 

\begin{pro}\label{p:blockofO} The category $\fg\mod_{\lambda}^B$ is compactly 
generated by the Verma 
modules $M_{w \cdot \lambda}$ for $w \in W$. In particular, the category $\fg\mod_\lambda$ only depends on the Weyl group dot orbit of $\lambda$. 
\end{pro}


\begin{proof}  For a simple reflection $s$ of $W$, write $j_{s, !}$ for the 
corresponding standard object and $j_{s, *}$ for the corresponding costandard 
object of $D(B \backslash G / B)$. For a $D(G)$-module $\scc$, let us denote the underlying binary product of the action of $D(B \backslash G / B)$ on $\scc^B$ by 
\begin{equation} \label{e:actonBinvs}
     - \star - : D(B \backslash G / B) \otimes \scc^B \rightarrow \scc^B, \quad \quad M \boxtimes c \mapsto M \star c. 
\end{equation}
We will deduce the proposition from 
the 
following lemma. 
	
\begin{lem} \label{l:intops} Fix a simple reflection $s$ and a weight $\nu \in 
\Lambda$.  
	We have 
	\begin{align} \label{e:stb1} &j_{s, !} \star M_{\nu} \simeq M_{s \cdot 
		\nu}  \quad \quad \quad  \text{ if $\langle 
		\halpha_s, \nu + \rho \rangle \leqslant 0.$} 
	\\ \label{e:stb2}&j_{s, *} \star M_{\nu} \simeq M_{s \cdot \nu} \quad \quad 
	\quad 
	\text{if 
	$\langle \halpha_s, \nu + \rho \rangle \geqslant 0$}. 	\end{align}
\end{lem}
\begin{re} The result of the lemma is well known, albeit typically phrased 
	in the language of 
	twisting functors. However, the fact that it tautologically follows from 
	the case of $\fsl_2$ via categorical equivariance properties, as we show 
	below, appears to be new.  
\end{re}
\begin{proof}  Write $P_s$ 
for the parabolic associated to $s$, and $G_s$ for 
	its Levi quotient, which is of semisimple rank one. Writing $B_s$ for its 
	induced Borel subgroup, by the $D(B_s \backslash G_s / 
	B_s)$-equivariance 
	of 
	\[
	\on{pind}: \fg_s\mod^{B_s} \rightarrow \fg\mod^B,
	\]
	it suffices to establish the lemma for $G_s$. In particular, for the 
	remainder 
	of the lemma when we 
	refer to the Verma modules $M_\nu$ and their simple quotients, we are 
	referring 
	to those for $G_s$. 
	
	As $j_{s,!}$ and $j_{s, *}$ are inverse functors, it is straightforward to 
	see that \eqref{e:stb1} and \eqref{e:stb2} are equivalent. To prove 
	\eqref{e:stb1}, recall the triangle 
	\begin{equation} \label{e:dtv}
	\underline{k}  \rightarrow j_{e, *} \rightarrow j_{s, !} \xrightarrow{+1},
	\end{equation}
	where $\underline{k}$ denotes the constant sheaf on $B_s 
	\backslash 
	G_s 
	/ B_s$.For any 
	$D(G_s)$-module $\sC$, one has an adjunction 
	\[
	\on{Oblv}: \sC^{G_s} \rightleftarrows \sC^{B_s}: \on{Av}_*
	\]
	and it straightforward to see that the counit $\on{Oblv} \on{Av}_* 
	\rightarrow 
	\on{id}$ identifies with convolution by $\underline{k} \rightarrow j_{e, 
	*}$ 
	above. 
	
	Let us prove \eqref{e:stb1}. In the singular case, i.e. if $\langle 
	\halpha_s, 
	\nu + \rho \rangle = 0$, we are done from \eqref{e:dtv} as $\on{Av}_* 
	M_\nu$ 
	vanishes. If $\langle \halpha_s, \nu + \rho \rangle < 0$, the desired 
	conclusion follows from rotating 
	\eqref{e:dtv} to obtain a triangle 
	\begin{equation} \label{e:dtnr}
	M_\nu \rightarrow j_{s, !} \star M_\nu \rightarrow L_{s \cdot \nu} \otimes 
	\Hom_{\fg_s\mod^{B_s}}(L_{s \cdot \nu}, M_{\nu}) [1] \xrightarrow{+1},
	\end{equation}
	and noting that $j_{s, !} \star M_\nu$ must be 
	indecomposable by the invertibility of $j_{s, !}$. \end{proof}	
		
	Write $\sS$ for the 
	subcategory of $\fg\mod^B$ generated the Verma modules $M_{w \cdot 
	\lambda}$, for $w \in W$. Using Lemma \ref{l:intops}, it is straightforward 
	to see 
	that $\sS$ is preserved by the action of $D(B \backslash G / B)$, 
	and 
	moreover is generated by $M_\lambda$ as a $D(B \backslash G/ 
	B)$-module. We 
	are 
	therefore done by Lemma \ref{l:gencomps}. \end{proof}

\begin{cor}\label{c:gmodrad} The inclusion of $\fg\mod_{\lambda}$ admits a 
 right adjoint $i^!$. 
\begin{equation} \label{e:adjwinv}
i_!: \fg\mod_\nu \rightleftarrows \fg\mod: i^!.
\end{equation}
\end{cor}

\begin{re} We remind that $i^!$, as with any adjoint of a $D(G)$-equivariant functor,  inherits a datum of {\em strict} $D(G)$-equivariance, cf. Section \ref{ss:catreps}.
\end{re}

\begin{proof} By Theorem \ref{t:cuff}, the inclusion $\fg\mod^{B\hspace{-.6mm} \on{-gen}} 
\rightarrow \fg\mod$ admits a right adjoint, so it is enough to 
furnish a right adjoint for 
\[
  \fg\mod_{\lambda} \rightarrow \fg\mod^{B\hspace{-.6mm} \on{-gen}}.
\]
Equivalently, by Proposition \ref{p:fftensoringup}, we must produce a 
$D(B 
\backslash G / B)$-equivariant right adjoint to $\fg\mod_{\lambda}^B 
\rightarrow \fg\mod^B$. However, from Proposition \ref{p:blockofO} it is 
clear the inclusion preserves compact objects, and hence admits a continuous 
right adjoint $i^!$. Its equivariance follows from the semi-rigidity of 
$D(B 
\backslash G / B)$, cf. Lemma 3.5 of \cite{bznadler}. \end{proof}

\subsubsection{} \label{sss:defcencharweakinv}Let us establish some notation 
for later use. For a 
$D(G)$-module $\sC$, let us 
define 
\[
\sC^{G, w, \nu} := \Hom_{D(G)\mod}( \fg\mod_\nu, \sC).
\]
By concatenating \eqref{e:gmodweakinvs} and \eqref{e:adjwinv}, we deduce an 
adjunction with fully faithful left adjoint
\begin{equation} \label{e:adjweakinv}
i_!: \sC^{G, w, \nu} \rightleftarrows \sC^{G, w}: i^{!}.
\end{equation}

\subsubsection{} Recall that $\fg\mod$ is self-dual as an abstract category, cf. Section \ref{ss:liealgreps} for our normalizations.  
As a final 
property, let us describe how the above categories behave under this
duality. To state this, recall that given an adjunction between dualizable 
categories 
\[
   F: \sC \rightleftarrows \sD: G,
\]
one canonically obtains a dual adjunction between the dual functors
\[
   G^\vee: \sC^\vee \rightleftarrows \sD^\vee: F^\vee. 
\]
\begin{pro}\label{p:gmodduals} For any $\lambda$,  duality
interchanges the adjunctions 
	\[
	    i_!: \fg\mod_\lambda \rightleftarrows \fg\mod: i^! \quad \text{and} 
	    \quad i_!: \fg\mod_{-\lambda - 2 \rho} \rightleftarrows \fg\mod: i^!.
	\]
\end{pro}

\begin{proof} The dualizability of $\fg\mod_\lambda$ follows from Corollary \ref{c:gmodrad}. For a dualizable category $\sC$, let us write $\sC^c$ for its 
full (non-cocomplete) subcategory of compact objects, and let us denote 
by 
	\[
	   \mathbf{D}: \sC^{c} \simeq (\sC^\vee)^{c, op}
	\]
the associated identification of compact objects.

Recall that by Shapiro's lemma the dual of a Verma module is again a (cohomologically shifted) Verma 
module, namely 
\[
     \mathbf{D} M_\lambda \simeq M_{-\lambda - 2 \rho} \otimes \ell,
\]
where $\ell$ is a cohomologically graded line, see for example Lemma 
9.8 of  \cite{lpw}. In particular, it follows that 
\[
      \fg\mod_{\lambda} \subset \fg\mod_{-\lambda - 2 \rho}^\vee \quad 
      \text{and} \quad \fg\mod_{-\lambda - 2 \rho} \subset 
      \fg\mod_{\lambda}^\vee.
\]
By dualizing each of the two inclusions and comparing with the other, we deduce 
both are equalities. 
\end{proof}


	%
	%


\section{Localization I: The sign representation and the Steinberg module}
\label{s:locI}

In this section, we prove a localization theorem which embeds a maximally 
singular 
category of Lie algebra representations into the category of nondegenerate 
Whittaker sheaves. 

\subsubsection{} The structure is as follows. After introducing notation, we 
construct a functor of global sections via the universal properties of the 
appearing $D(G)$-modules. To see it provides the desired fully faithful 
embedding, by the general results of Section 
\ref{s:hckinv} it is enough to do so on the categories of Borel equivariant 
objects. 

In the 
present situation, the latter categories are copies of $\on{Vect}$. 
That is, said informally,  both modules have a 
`unique highest weight vector, up to scaling'. So, we simply need to check a 
certain 
endofunctor of $\on{Vect}$ is an equivalence, which is done in Theorem 
\ref{t:singloc}. Finally, we note the alternative, tautological localization as
weakly equivariant D-modules.

\subsection{Setup}

\subsubsection{} In this subsection, we introduce the two sides of the 
localization theorem. Let $G$ be a connected reductive group with a Borel 
subgroup 
$B$. Recall that $\Lambda$ denotes the character lattice of the abstract 
Cartan, and a fix a 
maximally singular weight $\lambda \in \Lambda$, i.e. a character $\lambda$ for 
which $W \cdot 
\lambda = \lambda$. Consider the associated category $\fg\mod_{\lambda}$ as in 
Section 
\ref{sss:gmodlambda}. Passing to its $B$ invariants, by Proposition 
\ref{p:blockofO} we have
\begin{equation}  \label{e:gmodline}
   \fg\mod_\lambda^B \simeq \on{Vect},
\end{equation}
with distinguished generator the Verma module $M_\lambda$. 

\subsubsection{} \label{sss:whitintro} Fix a Borel $B^-$ in general position 
with respect to $B$, and 
write $N^-$ for its unipotent radical.\footnote{This passage to an opposite is 
inessential, and 
is made to simplify notation in the proof.} Let $\psi$ be a generic additive 
character of $N^-$, cf. Section \ref{ss:twists}, and consider the associated category of 
nondegenerate Whittaker D-modules 
\begin{equation}
   D(G/N^-, \psi) := D(G)^{N^-, \psi}.  \label{e:whitc}
\end{equation}
This is canonically a $D(G)$-module via left convolution. 

Note that, if we 
write $\delta_e$ for the delta D-module at the identity of $G$, \eqref{e:whitc} 
is generated as a $D(G)$-module by 
$
  \delta_{\psi} :=  \on{Av}^{N^-, \psi}_* \delta_e. 
$
Passing to $B$ invariants, it is standard that 
\begin{equation} \label{e:whitline}
      D(G/N^-, \psi)^B \simeq \on{Vect}, 
\end{equation}
with canonical generator $\on{Av}_*^B \delta_{\psi}.$

\subsection{The localization theorem}

\subsubsection{} We first construct a $D(G)$-equivariant functor in the 
direction
$
D(G/N^-, \psi) \rightarrow \fg\mod_{\lambda}.
$
To do so, note that, by the identification of $(N^-, \psi)$ invariants and coinvariants, and tensor-hom adjunction, we have for any $D(G)$-module $\sC$ the equivalence
\begin{equation} \label{e:homoutwhit}
    \on{Hom}_{D(G)\mod}\big( D(G/N^-, \psi), \hspace{.5mm}\sC \big) \simeq \sC^{N^-, 
    -\psi}. 
\end{equation}
Explicitly, this equivalence is given, from left to right, by evaluation on 
$\delta_{\psi}$, cf. 
Section \ref{sss:whitintro}. Equivalently, from right to left, it is given by 
convolution with the specified Whittaker equivariant object. 

\subsubsection{} Let us write $\fn^-$ for the Lie algebra of $N^-$, and 
$k_{-\psi}$ for its character given by the differential of $-\psi$. 
Consider the universal Whittaker module 
\[
      \on{ind}_{\fn^-}^{\fg} k_{-\psi}  \in \fg\mod^{N^-, \psi}. 
\]
By \eqref{e:homoutwhit}, it gives rise to an equivariant functor 
$
   \mathring{\Gamma}: D(G/N^-, \psi) \rightarrow \fg\mod.
$
It will be convenient in what follows to work instead with a cohomological shift of $\mathring{\Gamma}$, namely
\[
    \Gamma := \mathring{\Gamma} \otimes \det \big( (\fg/\fb)^*[-1]\big).
\]
%

\subsubsection{} Recall that the inclusion $\fg\mod_{\lambda} \rightarrow 
\fg\mod$ admits an equivariant right adjoint $i^!$, cf. Corollary 
\ref{c:gmodrad}. Consider the composition 
\begin{equation} \label{e:singeloc}
     i^! \circ \Gamma: D(G/N^{\text{--}}, \psi) \rightarrow \fg\mod 
     \rightarrow 
     \fg\mod_\lambda. 
\end{equation}
As the target is generated by $B$ invariants, by Corollary \ref{c:rad} it factors through a map 
\begin{equation} \label{e:singloc}
    D(G/N^-, \psi)^{B\hspace{-.6mm} \on{-gen}} \rightarrow \fg\mod_\lambda, 
\end{equation}
which we presently show is the sought-for functor. 

\begin{theo} \label{t:singloc} The functor \eqref{e:singloc} is an equivalence. \end{theo}

\begin{proof} By Corollary \ref{c:checkonhwvs}, it is enough to show the 
induced map 
	\[
	    \on{Vect} \overset{\eqref{e:whitline}}\simeq D(B \backslash G / 
	    N^-, 
	    \psi) \rightarrow 
	    \fg\mod_{\lambda}^B \overset{\eqref{e:gmodline}}\simeq \on{Vect}
	\]
is an equivalence, i.e. is given by tensoring by a cohomologically graded 
line. To see 
this, we may replace $\Gamma$ by $\mathring{\Gamma}$, and we simply unwind, using the equivariance of 
the appearing functors and adjunction:
\begin{align*}
    \Hom_{\fg\mod_{\lambda}^B}( M_\lambda, \hspace{1mm} i^! \circ \mathring{\Gamma} 
    \circ  
    \on{Av}^B_{*} \hspace{.5mm}
    (\delta_{\psi})) & \simeq \Hom_{\fg\mod_{\lambda}^B}(M_\lambda, 
    \hspace{1mm} 
    \on{Av}^B_{*} \circ \hspace{1mm} i^! \circ \mathring{\Gamma} \hspace{.5mm} 
    (\delta_{\psi})) \\ & 
    \simeq 
    \Hom_{\fg\mod}(M_\lambda, \hspace{1mm} \mathring{\Gamma}\hspace{.5mm}(\delta_{\psi})) 
    \\ & \simeq 
    \on{Hom}_{\fg\mod} (M_\lambda, \hspace{1mm} \on{ind}_{\fn^-}^{\fg} 
    k_\psi).
 \intertext{Writing $\fb$ for the Lie algebra of $B$, and noting that 
 $\on{Res}_{\fg}^{\fb} \circ \on{ind}_{\fn^-}^{\fg} k_{-\psi} \simeq 
 U(\fb)$, we compute}
 & \simeq \on{Hom}_{\fb\mod}( k_\lambda, \hspace{1mm} 
 \on{Res}_{\fg}^{\fb} \circ 
 \on{ind}_{\fn^-}^{\fg} k_\psi)
 \\ & \simeq \det( \fb^*[-1]),
\end{align*}
as desired. 
\end{proof}

\begin{re} Plainly, $\Gamma$ is given by global sections on $G/N^-, \psi$. It admits a left adjoint $\on{Loc}$ given by localization on $G/N^-, \psi$. One can show both functors are $t$-exact. In particular, Theorem \ref{t:singloc} is equivalent to the statement that $\on{Loc}$ restricts to an equivalence $\fg\mod_\lambda \simeq D(G/N^-, \psi)^{B\hspace{-.6mm} \on{-gen}}$.
\end{re}

\begin{re} The geometric side of the localization theorem may be thought of as follows. The finite group of Lie type $G(\mathbb{F}_q)$ has a distinguished simple $k$-representation,  the Steinberg module, constructed as follows. Within the Whittaker model of $k$-valued functions $$\on{Fun}(G(\mathbb{F}_q)/ N^-(\mathbb{F}_q), \psi)$$ there is a unique $B(\mathbb{F}_q)$-invariant line. It transforms as the sign representation of the Hecke algebra, and generates the Steinberg module for the group. Using the formalism of Section \ref{s:hckinv}, we produced a sheaf-theoretic analogue, $D(G/N^-,\psi)^{B\hspace{-.6mm} \on{-gen}}$. So, a paraphrase of Theorem \ref{t:singloc} is that `a maximally singular block in $\fg\mod$ is the same as the Steinberg module'. 
\end{re}

\subsubsection{}\label{sss:defstwhit} We now slightly rewrite the target of the 
localization 
theorem, using the following.

\begin{defn} For a $D(G)$-module $\scc$, define its {\em 
Steinberg--Whittaker} invariants as 
	\[
	     \scc^{N^-, \psi,\St} := \on{Hom}_{D(G)\mod}( D(G/N^-, \psi)^{B\hspace{-.6mm} \on{-gen}}, 
	     \scc)
	\]
\end{defn}
From the tautological adjunction 
$
     D(G/N^-, \psi)^{B\hspace{-.6mm} \on{-gen}} \rightleftarrows D(G/N^-, \psi),
$
we deduce for any $D(G)$-module $\scc$ an adjunction 
\[
    \scc^{N^-, \psi, \St} \rightleftarrows \scc^{N^-, \psi},
\]
wherein the left adjoint is fully faithful. Using this, we will slightly 
rewrite 
$D(G/N^-, \psi)^{B\hspace{-.6mm} \on{-gen}}$. Namely, we will replace the appearing $B\hspace{-.6mm} \on{-gen}$, which 
refers to the 
action by 
left convolution, with Steinberg--Whittaker invariants for the action by right 
convolution. 

\begin{pro} \label{p:bgenstwhit} The inclusions into $D(G/N^-, \psi)$ yield a 
$D(G)$-equivariant 
equivalence 
	\[
	    D(G/N^-, \psi)^{B\hspace{-.6mm} \on{-gen}} \simeq D(G/N^-, \psi, \St).
	\]

\end{pro}

\begin{proof} For any $D(G)$-module $\scc$, we have by Corollary \ref{c:rad}
	\begin{align*}
	   \scc^{N^-, \psi,\St} & \simeq \on{Hom}_{D(G)\mod}(D(G/N^-, 
	   \psi)^{B\hspace{-.6mm} \on{-gen}}, 
	   \scc) \\ & \simeq \on{Hom}_{D(G)\mod}( D(G/N^-, \psi), \scc^{B\hspace{-.6mm} \on{-gen}}) 
	   \simeq 
	   D(N^-, \psi \backslash G / B) \underset{D(B \backslash G / B)} \otimes  
	   \scc^B.
	\end{align*}
	Applying this to $\scc = D(G)$, viewed as a $D(G)$-module via convolution 
	on the right, yields the proposition. 
\end{proof}

Before moving on, let us describe the simplest case in more detail. 

\begin{ex} Suppose that $G = T$ is a torus. In this case, we have explicitly 
that 
	$
	    \sC^{N^-, \psi, \St} \simeq \sC^{T\hspace{-.6mm} \on{-mon}},
	$
cf. Example \ref{ex:moninvs} for the definition of the latter. In particular, 
Proposition \ref{p:bgenstwhit} reduces to the observation 
$$
        D(T\hspace{-.6mm} \on{-mon} \backslash T) \simeq D(T / T\hspace{-.6mm} \on{-mon}),
$$
since both are the full subcategory of $D(T)$ compactly generated by the 
constant sheaf. \end{ex}

\subsection{The weakly equivariant realization}

\subsubsection{} Recall the tautological identification of Lie algebra 
representations and weakly equivariant D-modules on the group
$$\Gamma^G: D(G)^{G, w} \simeq \fg\mod,$$
cf. Section \ref{ss:liealgreps}. In the remainder of this section, we identify the full 
subcategory of 
$D(G)^{G, 
	w}$ corresponding to $\fg\mod_\lambda$, and more generally do so for $\fg\mod_\nu$, for any $\nu \in \Lambda$. 

To state the answer, recall that in Section 
\ref{sss:defcencharweakinv} to any $D(G)$-module $\scc$ and character $\chi \in 
\Lambda$ we attached a certain full subcategory of the weak invariants 
$$ \sC^{G, w, \chi} \subset  \sC^{G, w}.$$
\begin{pro}\label{p:rhoshift} For any $\nu \in \Lambda$, the composition $\fg\mod_\nu \subset \fg\mod \simeq D(G)^{G, w}$
	factors through a $D(G)$-equivariant
	equivalence 
	\begin{equation} \label{e:winvlrr}
	\fg\mod_\nu \simeq D(G)^{G, w, -\nu -2\rho}.
	\end{equation}
\end{pro}

\begin{proof} Consider the commutative diagram 
	\begin{equation} \label{e:bbigdia}
	\xymatrix{\fg\mod_{\nu} \ar@<.5ex>[r] \ar@{-}[d]^{\wr} & D(G)^{G, w} 
		\ar@<.5ex>[l]  \ar@{-}[d]^{\wr}
		\\ 
		\Hom_{D(G)\mod}(D(G), \fg\mod_{\nu}) \ar@<.5ex>[r] 
		\ar@{-}[d]^{\wr} & 
		\Hom_{D(G)\mod}(D(G), D(G)^{G, w}) \ar@<.5ex>[l]  
		\ar@{-}[d]^{\wr}
		\\ 
		\Hom_{D(G)\mod}(\fg\mod_{\nu}^\vee, D(G)) \ar@<.5ex>[r] & 
		\Hom_{D(G)\mod}(D(G)^{G, w}, D(G)) \ar@<.5ex>[l]  
		\ar@{-}[d]^{\wr}
		\\ 
		& {}^{G, w}D(G)},
	\end{equation}
	where to pass between the second and third rows we use for dualizable 
	$D(G)$-modules $\sM$ and $\sN$ the 
	tautological equivalence
	$ 
	\Hom_{D(G)\mod}(\sM, \sN) \simeq \Hom_{D(G)\mod}(\sN^\vee, 
	\sM^\vee). 
	$

	It is straightforward to see that the composition of the right column of 
	equivalences in \eqref{e:bbigdia} is the equivalence induced by inversion $g 
	\mapsto g^{-1}$ on the group
	\[
	\on{inv}_*: D(G)^{G, w} \simeq {}^{G, w}D(G).
	\]

	Accordingly, to prove the proposition, it is enough to show that the 
	standard 
	self-duality of $D(G)^{G, w}$, i.e. the one induced by Verdier duality and the identification of weak invariants and coinvariants, induces an equivalence
	\begin{equation}\label{e:dualsrhoshiftt}
	\fg\mod_{\nu}^\vee \simeq \fg\mod_{-\nu - 2 \rho}. 
	\end{equation}
	%
	%
	%
	%
	%
	Under $\Gamma^G$, the self-duality of $D(G)^{G,w}$ interchanges with 
	the 
	self-duality of $\fg\mod$ given by Lie algebra cohomology, so we are done by Proposition \ref{p:gmodduals}. \end{proof}
	%
	%
	%
	%
	%
	%


	%
	%

\section{parabolic induction for groups and Hecke categories} 
\label{s:pind}
To prove the general case of the localization theorem, we will reduce to the 
maximally singular case using parabolic induction for categorical 
representations. In this section, we first collect 
basic definitions and properties of parabolic induction functors for groups and 
Hecke algebras. We then establish a basic compatibility between these induction 
functors, and deduce a useful corollary on submodules generated by 
Borel-invariant vectors.

The reader mostly interested in representation-theoretic applications may wish to skim Section \ref{ss:defpind} for the basic definitions and read the comments following Corollary \ref{c:pindgben} before proceeding on to Section \ref{s:locII}.


\subsection{Definitions and first properties}\label{ss:defpind}
\subsubsection{} Let $G$ and $B$ be as in Section \ref{s:pn}. Fix a standard 
parabolic subgroup $B \subset P \subset G$.\footnote{Note that for a loop group $G$, these are sometimes called instead {\em parahoric} subgroups.} 
Denote its 
prounipotent radical and Levi quotient by $N_P$ and $L$ respectively. One then has the standard correspondence of groups
\[
 L \leftarrow P \rightarrow G.
\]
\begin{defn}The functor of {\em parabolic induction} is given by the 
composition 
\begin{equation} \label{e:pind}
 \msf{pind}_P^G: D(L)\mod \xrightarrow{\msf{res}_L^P} D(P)\mod 
 \xrightarrow{\msf{ind}_P^G} 
 D(G)\mod.
\end{equation}
\end{defn}

By its definition, parabolic induction is left adjoint to the {\em Jacquet}, or 
{\em parabolic restriction }
functor 
\begin{equation} \label{e:defpres}
\on{pres}_G^P: D(G)\mod \xrightarrow{\msf{res}_G^L} D(P)\mod 
\xrightarrow{\msf{inv}^{N_P}} 
D(L)\mod.
\end{equation}

\begin{re} \label{r:dimthry} Parabolic induction also right adjoint to the 
Jacquet functor, 
although in affine type 
this depends on the choice of a dimension theory on $G$ 
to normalize the equivalence $D(G) \simeq D(G)^\vee$, cf. Section \ref{ss:dualsdmods}.
\end{re}

When $P$ and $G$ are clear from context, we will refer to the functors simply 
as $\on{pind}$ and $\on{pres}$. It is straightforward to see that both functors 
are 
$\DGCat$-linear and commute with limits and colimits. Similarly, $\on{pres}$ 
canonically commutes with duality, and the same holds for $\on{pind}$ with the 
caveat for $G$ of affine type as in Remark \ref{r:dimthry}.



	%


\subsubsection{} Let us pass to Hecke categories. Writing $B_L \simeq B/N_P$ 
for the induced Borel subgroup of $L$, one has a correspondence
\[ B_L \backslash L / B_L \leftarrow B \backslash P / B \rightarrow B 
\backslash G / B.
\]
This induces a correspondence of monoidal functors
\[
    D(B_L \backslash L / B_L) \leftarrow D(B \backslash P/B) \rightarrow D(B \backslash G / B),
\]
wherein the leftward arrow, by the prounipotence of $N_P$, is an equivalence, and the rightward arrow, as it is a pushforward along a closed embedding, is fully faithful. 
\begin{defn} The functor of {\em parabolic induction} is the
composition 
\begin{equation}
\label{e:pindh}
\msf{pind}_{\sH_L}^{\sH_G}: D(B_L \backslash L / 
B_L) \mod \simeq D(B \backslash P / B)\mod \xrightarrow{\msf{ind}} 
D(B 
\backslash G / B)\mod.
\end{equation}
\end{defn}

As before, we have that $\on{pind}_{\sH_L}^{\sH_G}$ is left adjoint to the 
{\em parabolic restriction} functor
	\begin{equation} \label{e:defpresh}
	 \on{pres}_{\sH_G}^{\sH_L}: D(B \backslash G  / B ) \xrightarrow{ \msf{res}} D(B 
	\backslash P 
	/ B)\mod 
	\simeq D(B_L \backslash L / B_L)\mod.
	\end{equation}
\begin{re} It follows from Corollary \ref{c:rad} that $\on{pind}_{\sH_L}^{\sH_G}$ is also canonically right 
adjoint to $\on{pres}_{\sH_G}^{\sH_L}$. 	
\end{re}

	%
	%
When it is clear from context that we are referring to the Hecke categories associated to $P$ and $G$, we will refer to the functors simply as $\on{pind}$ and $\on{pres}$. It is straightforward to see that $\on{pind}$ and $\on{pres}$ are 
$\DGCat$-linear, and commute with limits, colimits, and duality.

\subsection{Commutation of parabolic induction and Borel invariants}

\subsubsection{} Having now discussed parabolic induction functors for groups 
and Hecke categories, we now establish their compatibility. Informally, the 
following theorem states that taking Borel invariant vectors commutes with 
parabolic induction.

\begin{theo}\label{t:invpind} The following diagram commutes up to natural 
equivalence:
	\begin{equation} \label{e:pindcomm}
\begin{tikzcd}
  D(L)\mod \arrow[{rr}, "\inv^{B_L}"] \arrow[dd, "\on{pind}_L^G"'] &&
  D(B_L 
  \backslash 
  L / 
  B_L)\mod 
  \arrow[dd, "\on{pind}_{\sH_L}^{\sH_G}"] \\ \\
  D(G)\mod \arrow[rr, "\inv^B"] && D(B \backslash G / B)\mod.
\end{tikzcd}	
	\end{equation}

\end{theo}

\begin{proof} Note that the appearing functors are $\DGCat$-linear and commute with 
colimits. In particular, the composition $\on{pind}_{\sH_L}^{\sH_G} \circ \on{inv}^{B_L}$ is 
given by tensoring over $D(L)$ with the $D(B\backslash G / B) \otimes 
D(L)$-module 
	\[
	   D(B \backslash G / B) \underset{D(B\backslash P /B)}{\otimes} D(B_L 
	   \backslash L) \simeq D(B \backslash G / B) \underset{D(B\backslash P 
	   /B)}{\otimes} D(B \backslash P / N_P). 
	\]
Similarly, the composition $\on{inv}^B \circ \on{pind}_L^G$ is given by tensoring 
over $D(L)$ with the $D(B \backslash G / B) \otimes D(L)$-module 
\[
  D(B \backslash G) \underset{D(G)}{\otimes} D(G/N_P) \simeq D(B \backslash G / 
  N_P). 
\]	
%
		
	%
%
	%

		%

        Therefore, it is enough to show the natural convolution morphism  
		\begin{equation} \label{e:pindBinvunivcase} D(B \backslash G / 
		B) 
		\underset{D(B  
		\backslash P / B)}{\otimes} 
		D(B\backslash P/N_P) \rightarrow D(B \backslash G/N_P). 
		\end{equation}
		is an equivalence. As this identifies with the counit map for the 
		action of $D(L)$ 
		on 
		$D(B \backslash G / N_P)$ via right convolution, by Theorem 
		\ref{t:cuff} it is fully faithful. 
		
		It remains to prove 
		the essential 
		surjectivity of \eqref{e:pindBinvunivcase}. For this, let $W$ denote 
		the 
		Weyl group of $G$, and for each $w \in 
		W$ consider the Schubert cell \[C_w := BwB \subset G.\]Write $j_w: 
		B \backslash C_w \rightarrow B \backslash G$ for the associated locally 
		closed embedding. In view of the 
		Cousin filtration and 
		the 
		already established fully faithfulness, it 
		suffices to 
		show that $j_{w,*}\mathscr{M}$ belongs to the essential image of 
		\eqref{e:pindBinvunivcase} for any 
		$w 
		\in W$ and 
		any $\mathscr{M}$ in $D(B \backslash C_w/N_P)$.
		
		In fact, we will obtain all such $\mathscr{M}$ by convolution from 
		D-modules on a single Schubert cell of $L$, as follows. Write $w = uy$ 
		where $y \in W_L$ and $u$ is the minimal length 
		representative 
		of the coset $wW_L$. Then, writing $\ell$ for the standard length function on the Coxeter group $W$,  $\ell(w) = \ell(u) + \ell(y)$ and convolution 
		yields an isomorphism  \begin{equation}\label{e:step1} C_u 
		\stackrel{B}{\times} C_y \longrightarrow  
		C_w. \end{equation}In addition, consider the tautological projection 
		\begin{equation*} 
		   \pi: C_u \stackrel{B}{\times} C_y  \rightarrow B \backslash C_y. 
		\end{equation*}
		By our assumption on $u$ and $y$, this induces an isomorphism after 
		quotienting by $B$ and $N_P$, i.e. 
		\begin{equation} \label{e:step2}
		 \pi:  B \backslash C_u \stackrel{B}{\times} C_y  /N_P \simeq B 
		\backslash C_y / N_P. 
		\end{equation}
		Finally, note the resulting composition 
		\begin{equation}\label{e:itsjustconv}
		    D(B \backslash C_y / N_P) \overset{\eqref{e:step2}}{\simeq}   D(B 
		    \backslash C_u \stackrel{B}{\times} C_y / N_P)   
		    \overset{\eqref{e:step1}}{\simeq}  D(B \backslash C_{w} / N_P) 
		    \xrightarrow{ j_{w, *}} D(B \backslash G / N_P)
		\end{equation}
		is given by convolution on the left with the costandard object $j_{u, 
		*}$. In particular, the image of \eqref{e:itsjustconv} lies in the 
		image of \eqref{e:pindBinvunivcase}, as desired. 	\end{proof}

	
	\begin{re} By 	
	passing to the right 
	adjoints of the functors appearing in 
	\eqref{e:pindcomm}, one obtains a similar commutative diagram for parabolic restriction. 
	\end{re}

\subsection{Compatibility with generation by Borel invariants}

\subsubsection{} \label{s:binvs}We now record a consequence of Theorem 
\ref{t:invpind} which 
will be of use in what follows. Recall that to a $D(L)$-module $\scc$ or a
$D(G)$-module $\sD$ one has the associated submodules generated by their Borel 
invariants
\[
      \scc^{B_L\hspace{-.6mm} \on{-gen}} \quad \text{and} \quad \sD^{B\hspace{-.6mm} \on{-gen}},
\]
respectively, cf. Section \ref{s:genbykinvs} for their definition.

\subsubsection{}The following corollary, said informally, asserts that the 
operation of taking 
the submodule generated by Borel invariants commutes with parabolic induction 
and restriction. 

\begin{cor}   \label{c:pindgben}There is a natural isomorphism of functors 
$D(L)\mod \rightarrow 
D(G)\mod$:
\begin{equation} \label{e:binvind}
    \on{pind}_L^G (\scc^{B_L\hspace{-.6mm} \on{-gen}}) \simeq  (\on{pind}_L^G \scc)^{B\hspace{-.6mm} \on{-gen}}.
\end{equation}
Similarly, there is a natural isomorphism of functors $D(G)\mod \rightarrow 
D(L)\mod$:
\begin{equation} \label{e:binvres}
   \on{pres}_G^L( \sD^{B\hspace{-.6mm} \on{-gen}}) \simeq ( \on{pres}_G^L \sD)^{B_L\hspace{-.6mm} \on{-gen}}. 
\end{equation}    
\end{cor}

\begin{proof} We begin with \eqref{e:binvind}. Note that the left hand side of 
\eqref{e:binvind} is given by tensoring over $D(L)$ with the $D(L) \otimes 
D(G)$-module 
	\[
	     D(G) \underset{D(P)}{\otimes} D(L/B_L) \underset{D(B_L \backslash L / 
	     B_L)}{\otimes} D(B_L \backslash L) \simeq D(G/B) \underset{D(B_L 
	     \backslash L / 
	     B_L)}{\otimes} D(B_L \backslash L).
	\]

	Similarly, the right hand side is given by tensoring 
	with 
	\begin{align*}
	       D(G/B) \underset{D(B \backslash G / B)}{\otimes} D(B \backslash G) 
	       \underset{D(G)}{\otimes} D(G/N_P)  &    \simeq D(G/B)	 
	       \underset{D(B 
	       	\backslash G / B)}{\otimes} D(B 
	       \backslash G \stackrel{G} \times G/N_P)
	       \\ 
	       & \simeq D(G/B)    \underset{D(B \backslash G / B)}{\otimes} D(B 
	       \backslash G / N_P).
	       	\intertext{To identify the two, we may apply 
	       	\eqref{e:pindBinvunivcase} to $D(B \backslash G / N_P)$ to further rewrite the right hand side as }
	       \overset{\eqref{e:pindBinvunivcase}}&{\simeq}  D(G/B) \underset{D(B 
	       \backslash G / B)}{\otimes} D(B \backslash G / B) \underset{D(B_L 
	       	\backslash L / 
	       	B_L)}{\otimes} D(B_L \backslash L)
       	\\ & \simeq  D(G/B) \underset{D(B_L 
       		\backslash L / 
       		B_L)}{\otimes} D(B_L \backslash L),
	\end{align*}
	as desired. The identification \eqref{e:binvres} now follows by either 
	adjunction or a similar analysis. \end{proof}

Corollary \ref{c:pindgben} in particular implies that generation by Borel 
invariant vectors, cf. Definition \ref{d:genkinvs}, survives parabolic 
induction and restriction. So, we may control such categorical representations 
of groups via Hecke categories. Let us now apply this to localization theory.

\section{Localization II: the parabolic induction of the Steinberg module and 
Lie 
algebra representations}
\label{s:locII}

Let $G$ and $B$ be as in Section \ref{s:catliereps}. Recall the set of integral weights  $\Lambda$, cf. Sections \ref{ss:grpslies} and \ref{ss:liealgreps}, and fix $\lambda \in 
\Lambda$ such that 
the corresponding Verma module $M_\lambda$ is {\em irreducible}. 

\subsubsection{} In this 
section, we prove localization theorems for the corresponding category 
\begin{equation}
\label{e:gmodlambda} \fg\mod_{\lambda}. 
\end{equation}
We will do so by identifying this with the parabolic induction of a 
similar category for a Levi. The latter will be, by construction, maximally 
singular. Applying parabolic induction to its two geometric realizations from 
Section \ref{s:locI} will provide the desired localization of 
\eqref{e:gmodlambda}. 

\subsection{The parabolic induction of the Steinberg module}
\label{ss:parindst}
\subsubsection{} 
Note that, since $M_\lambda$ is irreducible, by the representation theory of 
$\fsl_2$ it follows that the weight 
$\lambda$ is anti-dominant, i.e. for any simple coroot $\halpha$ we have 
\begin{equation} \label{e:ineq}
  \langle \halpha, \lambda + \rho \rangle \leqslant 0.
\end{equation}

\begin{re}In our present setting the reverse implication holds too, 
but this is	a nontrivial theorem of Kac--Kazhdan in the affine case (at 
non-critical level) \cite{kk}. We emphasize that we do not need this fact. 
\end{re}

In particular, the stabilizer of $\lambda$ under the dot action of $W$ is 
a parabolic subgroup $W_P$ corresponding to a proper parabolic $P \supset B.$ 
As in 
Section \ref{s:pind},  let us write $N_P$ for the prounipotent radical of $P$ 
and $L$ for its Levi quotient.

\subsubsection{} Let us write $\fl$ for the Lie algebra of $L$, and consider 
the corresponding $D(L)$-module  
\[
   \fl\mod_{\lambda}.
\]
By construction, it is maximally singular, i.e. falls into the regime of 
Section \ref{s:locI}. As in Section \ref{s:pind}, we may consider its parabolic 
induction from $L$ to $G$:
\begin{equation}
\label{e:pindlmod}
   \on{pind}_L^G(\fl\mod_\lambda).
\end{equation}

We now use the results of Section \ref{s:locI}, along with a compatibility of 
Borel invariants and parabolic induction established in Section \ref{s:pind}, 
to rewrite \eqref{e:pindlmod} geometrically. In the following proposition, we 
write $\rho_L$ for 
the half sum of positive roots of $L$. In addition, for a generic character 
$\psi$ of $N_L$, we denote by the same letter its inflation to $N$, i.e. the 
composition 
\[
     N \rightarrow N/N_P \simeq N_L \xrightarrow{\psi} \mathbb{G}_a.
\] 
In particular, for a $D(G)$-module $\scc$, we may consider its 
Steinberg--Whittaker invariants
\[
   \scc^{N, \psi, \St} := (\scc^{N_P})^{N_L, \psi, \St}, 
\]
cf. Section \ref{sss:defstwhit} for the definition the Steinberg--Whittaker 
invariants of a $D(L)$-module. 

\begin{ex} \label{ex:moninvsBig}Note that if $L$ is the abstract Cartan $T$, 
the above definition simply recovers the 
monodromic invariants, i.e. 
	\[
	   \sC^{N, \psi, \St} \simeq \sC^{B\hspace{-.6mm} \on{-mon}},
	\]
cf. Example \ref{ex:moninvs} for the definition of the latter. 	\end{ex}

\begin{pro}\label{p:3realz} There are canonical $D(G)$-equivariant equivalences 
	\[
	    \on{pind}_L^G(\fl\mod_{\lambda}) \simeq D(G/N, \psi, \St) \simeq 
	    D(G/N_P)^{L, w, -\lambda - 2 
	    \rho_L} . 
	\]
\end{pro}

\begin{proof} In Section \ref{s:locI}, we saw $D(L)$-equivariant equivalences 
	\[
         \fl\mod_{\lambda} \simeq D(L/N_L, \psi, \St) \simeq D(L)^{L, w, 
         -\lambda - 2 \rho_L}. 	
	\]
In particular, we may parabolically induce these to obtain equivalences 
\[
   \on{pind}(\fl\mod_{\lambda}) \simeq \on{pind}(D(L/N_L, \psi, \St)) 
   \simeq 
   \on{pind}(D(L)^{L, w, 
	-\lambda - 2 \rho_L}).
\]
To rewrite the middle term, we first note the $\DGCat$-linearity and continuity 
of $(N_L, \psi, \St)$ invariants. Indeed, these follow via colocalization from the analogous assertions for 
$(N, \psi)$ invariants. Using this 
observation, we may rewrite the middle term as
\begin{align*}
\on{pind}(D(L/N_L, \psi, \St)) &\simeq   
 (\on{pind}(D(L)))^{N_L, \psi, \St} \\ & \simeq (D(G) \underset{D(P)} \otimes
 D(P/N_P))^{N_L, \psi, \St} \\  & \simeq D(G/N_P)^{N_L, \psi, \St} \\ & 
 \simeq D(G/N, \psi, \St).  
\end{align*}
Similarly, using the $\DGCat$-linearity and continuity of $(L, w, -\lambda - 2 
\rho_L)$ invariants, which again follow their analogues for $(L,w)$ invariants, we may rewrite the 
right hand term as 
\begin{align*}
\on{pind}( D(L)^{L, w, -\lambda - 2 \rho_L}) & \simeq D(G) 
\underset{D(P)} \otimes {D(P/N_P)^{L, w, -\lambda - 2 \rho_L}} \\ & \simeq 
D(G/N_P)^{L, 
w, -\lambda - 2 \rho_L},
\end{align*}
as desired. 
\end{proof}

\begin{re}  \label{r:whitbinvs} Note that by Corollary \ref{c:pindgben} and 
Theorem \ref{t:singloc}, 
we have 
\[ D(G/N, \psi, \St) \simeq D(G/N, \psi)^{B\hspace{-.6mm} \on{-gen}}. 
\]
\end{re}

\subsubsection{} Let us use Proposition \ref{p:3realz} to obtain a  
collection of compact generators for 
\begin{equation}
\label{e:pindlbinv}
     \on{pind}_L^G(\fl\mod_{\lambda})^B.
\end{equation}

To do so, write $\sM_e$ for the image of the Verma module $M_\lambda^L$ in 
$\fl\mod_{\lambda}^{B_L}$ 
under the unit of adjunction
\[
     \fl\mod_{\lambda}^{B_L} \rightarrow \on{pind}_L^G( \fl\mod_{\lambda})^B. 
\]
Recall that $j_{w, !} \star M_\lambda^L \simeq M_\lambda^L$ for $w \in W_P$, 
cf. Lemma \ref{l:intops}. In particular, we may consider the objects 
\[
     \sM_w  := j_{w, !} \star \sM_e, \quad \quad \text{for } w \in W/W_P.
\]

\begin{cor} The objects $\sM_w$, for $w \in W/W_P$, compactly generate 
\eqref{e:pindlbinv}. Moreover, we have the vanishing 
	\begin{equation} \label{e:vanwhit}
	    \on{Hom}(\sM_w, \sM_e) \simeq \begin{cases} k & \text{if } w = 
	    e, \\ 0 & \text{otherwise.} \end{cases} 
	\end{equation}
\end{cor}

\begin{proof} Proposition \ref{p:3realz} and Remark \ref{r:whitbinvs} yield an 
identification
	\[
\on{pind}_L^G(\fl\mod_\lambda)^B \simeq D(B \backslash G / N, \psi).
\]
By construction this identifies $\sM_e$ with the simple object of $D(B 
\backslash P / N, \psi)$. In particular the $\sM_w$, for $w \in W/W_P$, are identified with the 
standard objects of $D(B \backslash G/ N, \psi)$, which makes the claims 
manifest.  
\end{proof}

\subsection{The parabolic induction of the Steinberg module and Lie algebra 
representations}

\subsubsection{} To prove the localization theorems for $\fg\mod_\lambda$, it 
is enough to identify it with $\on{pind}_L^G(\fl\mod_\lambda)$. We now observe that 
there is a canonical functor in one direction.

\subsubsection{} Indeed, note that parabolic induction of Lie algebra 
representations is a $D(L)$-equivariant functor 
\begin{equation}
  \on{pind}_{\fl}^{\fg}: \fl\mod \rightarrow \fg\mod^{N_P}, \label{e:parindreps}
\end{equation}
cf. Section \ref{sss:eqindres} for the equivariance data. 

\begin{re}
In the affine case, if $\lambda$ is a weight of level $\kappa$, by this we mean the functor of level $\kappa$ parabolic induction
\[
   \fl\mod \rightarrow \hk\mod^{N_P} \hspace{1mm} \subset \hspace{1mm} \fg\mod. 
\]
\end{re}
\noindent By adjunction, \eqref{e:parindreps} induces a 
$D(G)$-equivariant functor
\[
    \on{pind}_L^G(\fl\mod) \rightarrow \fg\mod.
\]
As parabolic induction sends Verma modules to Verma modules, this restricts to 
a $D(G)$-equivariant functor 
\begin{equation} \label{e:pindreps}
  \on{pind}_L^G(\fl\mod_\lambda) \rightarrow \fg\mod_{\lambda}.
\end{equation}

\subsubsection{} We now show that the above yields the sought-for identification. 

\begin{theo} The functor \eqref{e:pindreps} is an equivalence. 
\label{t:gmodparind} 
\end{theo}

\begin{proof} As both sides of \eqref{e:pindreps} are generated by their Borel 
invariants (cf. Corollary \ref{c:pindgben} for the left hand side), it is enough 
to show the obtained $D(B \backslash G / B)$-equivariant functor 
\[
  \gamma: \on{pind}_L^G(\fl\mod_{\lambda})^B \rightarrow \fg\mod_{\lambda}^B
\] 
is an equivalence. 

We first claim it admits an equivariant right adjoint $\gamma^R$. Indeed, note 
that for any $w \in W/W_P$ we have 
\begin{equation}\label{e:mgoestom}
    \gamma( \sM_w) \simeq \gamma( j_{w, !} \star \sM_e) \simeq j_{w, !} \star 
    \gamma(\sM_e) \simeq j_{w, !} \star M_\lambda 
    \overset{\eqref{e:stb1}}{\simeq} M_{w \cdot \lambda}. 
\end{equation}
In particular, $\gamma$ sends compact objects to compact objects, and hence 
admits a continuous right adjoint, which is equivariant by the semi-rigidity of 
$D(B \backslash G / B)$. 

We next check that $\gamma$ is fully faithful. As $\sM_e$ generates 
\eqref{e:pindlbinv} as a $D(B \backslash G / B)$-module, it is enough to see 
the canonical map 
\begin{equation} \label{e:unitcheck}
     \sM_e \rightarrow \gamma^R \circ \gamma(\sM_e)
\end{equation}
is an equivalence. Indeed, for any $w \in W/W_P$, we unwind 
\begin{align}
    \Hom( \sM_w, \gamma^R \circ \gamma (\sM_e))  &\simeq \Hom( \gamma(\sM_w), 
    \gamma(\sM_e)) \overset{\eqref{e:mgoestom}}{\simeq} \Hom( M_{w \cdot 
    \lambda}, M_{\lambda}). 
\intertext{Recalling that $M_\lambda$ is simple, and in particular coincides 
with its 
contragredient dual $A_\lambda$, we may continue}  \label{e:vanMA}
   & \simeq \on{Hom}(M_{w \cdot \lambda}, A_\lambda) \simeq \begin{cases} k & 
    \text{if }  w = e \\ 0 & \text{otherwise;} \end{cases}
\end{align}
where the last vanishing is a (standard) consequence of the cofreeness of 
$A_\lambda$ over $N$. Comparing with \eqref{e:vanwhit}, it follows that 
\eqref{e:unitcheck} is an equivalence, as desired. 

It remains to see that $\gamma$ is essentially surjective. However, as the 
$M_{w \cdot \lambda}$ compactly generate the target, cf. Proposition 
\ref{p:blockofO}, we are done by \eqref{e:mgoestom}. 
\end{proof}

\subsubsection{} Concatenating Theorem \ref{t:gmodparind} and Proposition 
\ref{p:3realz}, we have shown the following. 

\begin{theo} \label{t:3realz} There are canonical $D(G)$-equivariant 
equivalences 
\begin{equation} \label{e:locequivs}
\fg\mod_{\lambda} \simeq D(G/N, \psi, \St) \simeq D(G/N_P)^{L, w, -\lambda - 
2 \rho_L}.
\end{equation}
\end{theo}

\subsubsection{} Let us explain how to deduce localization theorems at positive 
level as well.

\begin{cor} \label{c:3realz} There are canonical $D(G)$-equivariant 
equivalences 
	\[
	     \fg\mod_{-\lambda - 2 \rho} \simeq D(G/N, -\psi, \St) \simeq 
	     D(G/N_P)^{L, w, \lambda}. 
	\]
\end{cor}

\begin{proof} Recall that in Proposition \ref{p:gmodduals} we saw that 
	\[
	    \fg\mod_{\lambda}^\vee \simeq \fg\mod_{-\lambda - 2 \rho}.
	\]
	The corollary therefore follows by dualizing the categories appearing in 
	\eqref{e:locequivs}. 
\end{proof}

\begin{re} \label{r:longint} Consider the action of $T$ (supplemented with the 
loop rotation 
$\mathbb{G}_m$ in affine type, i.e. the Levi quotient of the automorphisms of a 
formal disk) on the additive characters of $N$. As $\psi$ and $-\psi$ lie in 
the same orbit, one may choose an identification
	\[
	     D(G/N, \psi, \St) \cong D(G/N, -\psi, \St).
	\]
The resulting equivalence $\fg\mod_\lambda \cong \fg\mod_{-\lambda - 2 \rho}$ 
behaves as a `long intertwining operator'. 
\end{re}

\section{The linkage principle for positive energy representations}
\label{s:linky}

The goal of this section is to prove a linkage principle for $\fg\mod^{B\hspace{-.6mm} \on{-gen}}$ and specialize it to a linkage principle for positive energy representations proposed by Yakimov in \cite{yakimov}. 

The proofs of this section do not rely on localization theory, but do use the definitions and results of Sections \ref{s:hckinv}, \ref{s:catliereps}, and \ref{s:pind}. The only material from this section used in its sequels is Theorem \ref{t:linky}, which will be applied to obtain the block decomposition for affine Harish-Chandra bimodules. So, the reader may wish to examine Theorem \ref{t:linky} and then proceed directly to Section \ref{s:someapps}.

\subsection{Positive energy representations}

\subsubsection{} 
Let $H$ be an almost simple, simply connected algebraic group, $H[\hspace{-.2mm}[z]\hspace{-.2mm}]$ its arc group, and $H(\!(z)\!)$ 
its loop group. Passing to their Lie algebras, let us write $\fh$ for the 
corresponding semisimple Lie algebra, $\fh[\hspace{-.2mm}[z]\hspace{-.2mm}]$ for the corresponding arc 
algebra, and $\hk$ for the corresponding affine Lie algebra 
at a fixed noncritical integral level. Recall that $\Lambda_f$ denotes the character lattice of the abstract Cartan of $H$. 

\subsubsection{} Associated to this data is the following category of {\em 
positive energy representations}. To introduce it, write $K$ for the first 
congruence 
subgroup of $H[\hspace{-.2mm}[z]\hspace{-.2mm}]$, i.e. its prounipotent radical, and consider the abelian category $$\hk\mod^{K, \heartsuit}$$of smooth $\hk$-modules on which the action of the Lie algebra of $K$ is integrable, i.e. locally nilpotent. Within it, consider the cocomplete Serre subcategory 
\[
       \per \subset \hk\mod^{K, \heartsuit}
\]
generated by parabolic inductions of $\fh$-modules with integral infinitesimal characters. Explicitly, if we write $Z(\fh)$ for the center of $U(\fh)$, $\per$ may be described as follows. 

\begin{lem} For a module $M$ in $\hk\mod^{K, \heartsuit}$,  the following are equivalent.

\begin{enumerate}
    \item The object $M$ lies in $\per$.
    
        \item For any subquotient $S$ of $M$, $Z(\fh)$ acts on $S^K$ locally finitely with integral eigenvalues. 

    \item For any splitting of $\fh$ into $\fh[\hspace{-.2mm}[z]\hspace{-.2mm}]$, $Z(\fh)$ acts on $M$ locally finitely with integral eigenvalues.

\end{enumerate}

\end{lem}

\begin{proof} It is straightforward to see that (1) implies (3), and (3) implies (2). To see that (2) implies (1), consider the chain of submodules
\begin{equation} \label{e:filt}
   0 = M^0 \subset M^1 \subset M^2 \subset \cdots \subset M 
\end{equation}
where for $i \geqslant 1$,  $M^i$ is defined inductively as the submodule of $M$ generated by vectors which are $K$-invariant modulo $M^{i-1}$. By assumption (2), each successive quotient $M^i/M^{i-1}$ lies in $\on{Pos}_\kappa$, and hence so does each $M^i$. It is therefore enough to show that the filtration is exhaustive, i.e. 
$
  \varinjlim_i M^i \simeq M. 
$
To see this, note that by the $K$-integrability of $M$, any element of $M$ lies in a finite dimensional $K$-submodule, and by the prounipotence of $K$, the latter is a finite successive extension of the trivial module. 
\end{proof}    
    
\begin{re}
   Note that the Sugawara energy operator $L_0$ acts locally finitely on objects of $\per$, with spectrum locally bounded from below, cf. the proof of Proposition \ref{p:poseng} below for a precise formulation. For $\fh \simeq \fsl_2$, and after relaxing the integrality assumption on the infinitesimal characters, this also characterizes $\per$. 
\end{re}

\subsubsection{} Let us write $W_f$ and $W$ for the  finite and 
affine Weyl groups of $H$, and consider the level $\kappa$ dot action of the latter on $\Lambda_f$. For an orbit of weights $\chi \in  W \backslash 
\Lambda_f$, and a infinitesimal character $\chi_f$ of $\fh$, let us say that $\chi_f$ 
{\em lies in} $\chi$ if the corresponding finite Weyl group orbit is 
contained in $\chi$. For an orbit $\chi$, consider the Serre subcategory 
\[
      \Pos_\chi \subset \per 
\]
consisting of modules $\sM$ with the property that, for every subquotient $\sS$ 
of $\sM$, the generalized eigenvalues of $Z(\fh)$ on $\sS^K$ lie in $\chi$. 

 The following linkage principle was proposed by Yakimov and 
will be proven in this section.

\begin{theo} \label{t:linky} The direct sum of inclusions over $\chi \in W 
\backslash \Lambda_f$ 
yields an equivalence of cocomplete 
abelian categories 
	\[
	\underset{\chi}\oplus   \hspace{.5mm}   \Pos_\chi \xrightarrow{\sim} \per.
	\]
I.e., every object $\sM$ of $\per$ canonically decomposes as $\underset \chi 
\oplus \hspace{.5mm} 
\sM_\chi$, where $\sM_\chi$ is its maximal submodule lying in $\Pos_\chi$.  
\end{theo}

\begin{re} The 
statement of Theorem \ref{t:linky} and argument below can be directly adapted to cover 
arbitrary infinitesimal 
characters at any noncritical level. To do so, one works also with $(B,\lambda)$-equivariant objects and the corresponding twisted affine Hecke categories, where $\lambda$ varies over character sheaves pulled back from the abstract Cartan; see for example \cite{ly} or \cite{qfle} for further discussion of such Hecke categories. \end{re}

\subsubsection{} This section is structured as follows. We will first prove a 
linkage principle for all of $\fg\mod^{B\hspace{-.6mm} \on{-gen}}$. We will then, in finite type, 
identify the latter category with $\fg$-modules with integral generalized 
infinitesimal characters. Finally, we will combine these two steps to deduce  
Theorem \ref{t:linky}. 
	
\subsection{The linkage principle}

\subsubsection{} In this subsection, $G$ will be as in Section \ref{ss:grpslies}. Recall that to any  Weyl group orbit $\chi \in W \backslash \Lambda$ 
we associated a full subcategory 
\[
\fg\mod_{\chi} \subset \fg\mod^{B\hspace{-.6mm} \on{-gen}},
\]
cf. Proposition \ref{p:blockofO}. Letting $\chi$ vary over the set of all orbits, the direct sum of the 
inclusion maps yields a tautological 
$D(G)$-equivariant functor 
\begin{equation} \label{e:linkp}
\underset{ \chi} \oplus  \fg\mod_{\chi} 
\rightarrow 
\fg\mod^{B\hspace{-.6mm} \on{-gen}}.
\end{equation}
We will now deduce a linkage principle from Kac--Kazhdan's calculation of the 
Shapovalov determinant. 

\begin{pro} \label{p:linkp} The map \eqref{e:linkp} is an equivalence. 
\end{pro}

\begin{proof} We claim it is enough to show the obtained map 
	\begin{equation} \label{e:linkpbinv}
	\underset{\chi}\oplus \fg\mod_{\chi}^B \rightarrow \fg\mod^B
	\end{equation}
	is an equivalence.\footnote{This assertion is a theorem due to Deodhar--Gabber--Kac \cite{dgk}; the ensuing should be understood as an alternative proof of their result.} Indeed, if so we have 
	\begin{align*}
	\fg\mod^{B\hspace{-.6mm} \on{-gen}} & \simeq D(G/B) \underset{D(B \backslash G / B)} 
	\otimes  
	\fg\mod^B \\ & \simeq D(G/B) \underset{D(B \backslash G / B)} \otimes  
	\underset{\chi}\oplus \hspace{1mm}\fg\mod_{\chi}^B \\ & \simeq 
	\underset{\chi}\oplus \hspace{1mm}
	\fg\mod_{\chi}. 
	\end{align*}

	Let us show \eqref{e:linkpbinv}, i.e. that for distinct $\phi, \chi$ and 
	modules $\sM \in \fg\mod_{\phi}^B$ and $\sN \in \fg\mod_{\chi}^B$, one has 
	\begin{equation} \label{e:vanblocks}
	\on{Hom}_{\fg\mod^B}(\sM, \sN) \simeq 0. 
	\end{equation}
	Let us first assume that $G$ is either of finite type or $G$ is affine and 
	we 
	are at a negative level. In these cases, by a theorem \cite{kk} of Kac--Kazhdan on 
	singular vectors in Verma modules (due to Shapovalov in finite type), there 
	is 
	a (unique) antidominant Verma module $M_\nu$ in $\fg\mod_{\chi}$.  
	
	As $\fg\mod_{\phi}$ is compactly generated, by Proposition \ref{p:blockofO} 
	we may assume that $\sN$ is a Verma module. As we may then write $\sN 
	\simeq 
	j_{w, !} \star M_\nu$ for some $w \in W$, by 
	\[
	\on{Hom}_{\fg\mod^B}(\sM, j_{w, !} \star M_\nu) \simeq 
	\on{Hom}_{\fg\mod^B}( 
	j_{w^{-1}, *} \star \sM, M_\nu),
	\]
	we may assume that $\sN \simeq M_\nu$. It suffices to  to take 
	$\sM$ to 
	be a Verma module  $M_\mu$, in which case 
	\[
	\on{Hom}_{\fg\mod^B}(M_\mu, M_\nu) \simeq \on{Hom}_{\fg\mod^B}(M_\mu, 
	A_\nu) 
	\simeq 0
	\]
	by \eqref{e:vanMA} and the assumption that $\phi$ and $\chi$ are distinct. 
	
	Finally, the case of $G$ affine and positive level follows formally from 
	the 
	case of negative level by duality, cf. Proposition \ref{p:gmodduals}. 
	\end{proof}

\subsection{The relation with infinitesimal characters}

\subsubsection{} Let $G$ be a connected reductive group. Recall that we 
defined, for a character 
$\lambda$ of its abstract Cartan, the category 
\[
\fg\mod_{\lambda} \subset \fg\mod
\]
as the $D(G)$-submodule generated by the Verma module $M_\lambda$. 
We now compare this with the traditional definition via infinitesimal characters.

\begin{re} Of course, the hypothesis that $G$ is reductive appears since in 
	the affine case there is not such a definition via infinitesimal characters, due 
	to the triviality of the center of the enveloping algebra. 	
\end{re}

\subsubsection{} Let us write $\chi$ for the character by which the center 
$Z(\fg)$ of 
$U(\fg)$ acts on $M_\lambda$. Consider the full subcategory 
\begin{equation}
\label{e:gencc1}
\fg\mod_{\chi} \subset \fg\mod
\end{equation}
consisting of objects on whose cohomology groups $Z(\fg)$ acts with generalized 
infinitesimal character $\chi$.

\subsubsection{}  For the convenience of the reader, we now collect 
some alternative presentations of this category.

The first alternative is via abelian categories. For a category $\sC$ with a 
$t$-structure, we denote by $\sC^\heartsuit$ 
the abelian category given by the heart of the $t$-structure. Consider the full 
subcategory
\[
\fg\mod^{\heartsuit}_\chi \subset \fg\mod^\heartsuit    
\]
consisting of objects with generalized infinitesimal character $\chi$. If we write 
$D(\sC^\heartsuit)$ for the corresponding dg-enhanced derived category of 
$\sC^\heartsuit$, we have a canonical functor 
\begin{equation}\label{e:gencc2}
D(\fg\mod_\chi^\heartsuit) \rightarrow D(\fg\mod^\heartsuit) \simeq 
\fg\mod. 
\end{equation}

The second alternative is via base change of categories. The centrality of 
$Z(\fg)$ in $U(\fg)$ gives rise to a tautological action of $Z(\fg)\mod$ on 
$\fg\mod$. Concretely, the associated action map  
\[
     Z(\fg)\mod \otimes \hspace{.5mm} \fg\mod \rightarrow \fg\mod.
\]
sends a $Z(\fg)$-module $\sM$ and a $\fg$-module $\sN$ to the tensor product 
$
      \sM \underset{Z(\fg)}{\otimes} \sN.
$ 

Let us write $Z(\fg)\mod_{\chi}$ for the full subcategory of 
$Z(\fg)\mod$ consisting of objects supported on the formal neighborhood of 
$\chi$, i.e. on which the maximal ideal $\mathfrak{m}_\chi$ acts locally nilpotently on their cohomology. There is a tautological $Z(\fg)\mod$ equivariant adjunction 
\begin{equation} \label{e:loccoh}
      i_*:  Z(\fg)\mod_{\chi} \rightleftarrows   Z(\fg)\mod: i^!_\chi,
\end{equation}
wherein $i_*$ is fully faithful. In particular, tensoring up with $\fg\mod$, we 
obtain a fully faithful embedding 
\begin{equation}
 \label{e:gencc3}    \fg\mod \underset{Z(\fg)\mod}\otimes 
 Z(\fg)\mod_{\chi} \subset \fg\mod.       
\end{equation}

All three categories coincide due to the flatness of $U(\fg)$ over $Z(\fg)$, 
as we spell out in the following. 

\begin{pro} The functors \eqref{e:gencc1}, \eqref{e:gencc2}, and 
\eqref{e:gencc3} induce equivalences 
\begin{equation} \label{e:gencc3real}
      \fg\mod_{\chi} \simeq D(\fg\mod_{\chi}^\heartsuit) \simeq  \fg\mod 
      \underset{Z(\fg)\mod}\otimes 
      Z(\fg)\mod_{\chi}.
\end{equation}
\end{pro}

\begin{proof} We begin by showing the map \eqref{e:gencc2} is a fully faithful 
embedding. This will be deduced from the analogous assertion for $Z(\fg)$ as 
follows. Consider the tautological commutative diagrams 
\[
\xymatrix{  \fg\mod_{\chi}^\heartsuit  \ar[d]_{\on{Oblv}} 
& \fg\mod^\heartsuit  \ar[l]_{i^!} \ar[d]^{\on{Oblv}}   & & &  
\fg\mod_{\chi}^\heartsuit  \ar[r]^{i_*} 
&\fg\mod^\heartsuit    \\ 
Z(\fg)\mod_{\chi}^\heartsuit &  
Z(\fg)\mod^\heartsuit \ar[l]_{i^!} & & & Z(\fg)\mod_{\chi}^\heartsuit 
\ar[u]^{\on{ind}} 
\ar[r]^{i_*} & Z(\fg)\mod^\heartsuit, \ar[u]_{\on{ind}} }
\]
where the left diagram is obtained from the right by passing to right adjoints.

We next note that $Z(\fg)\mod_\chi$ is the derived category of its heart. Indeed, recall that the injective envelope of the skyscraper sheaf $k_\chi$ supported at $\chi$ lies in $Z(\fg)\mod^\heartsuit_\chi$. As the category $Z(\fg)\mod_\chi^\heartsuit$ is locally Noetherian, and in particular arbitrary direct sums of injective objects are injective, it follows that the tautological map on bounded below categories
\begin{equation} \label{e:bb}
     D(Z(\fg)\mod^\heartsuit_\chi)^+ \rightarrow Z(\fg)\mod_\chi^+
\end{equation}
is fully faithful. It is straightforward to see that $Z(\fg)\mod_\chi$ is compactly generated by $k_\chi$. It therefore remains to see the same is true for $D(Z(\fg)\mod_\chi^\heartsuit)$. However, as $Z(\fg)\mod_\chi^\heartsuit$ has finite global dimension, we may identify its unbounded derived category with the homotopy category of injective complexes (and not merely its quotient by acyclic complexes), which makes the asserted compactness manifest.

Prolonging to the corresponding (dg-enhanced) derived categories, we obtain 
diagrams 
\begin{equation*}  
\xymatrix{  D(\fg\mod_{\chi}^\heartsuit)  \ar[d]_{\on{Oblv}} 
	& \fg\mod  \ar[l]_{Ri^!} \ar[d]^{\on{Oblv}}   & & &  
	D(\fg\mod_{\chi}^\heartsuit)  \ar[r]^{i_*} 
	&\fg\mod    \\ 
	Z(\fg)\mod_{\chi} &  
	Z(\fg)\mod \ar[l]_{Ri^!} & & & Z(\fg)\mod_{\chi} 
	\ar[u]^{\on{ind}} 
	\ar[r]^{i_*} & Z(\fg)\mod. \ar[u]_{\on{ind}} }
\end{equation*}
We claim these again commute. As the left diagram is again obtained from the 
right by passing to right adjoints, it is enough to argue for one diagram. 
However, on the right this is visible due to the $t$-exactness of the 
appearing functors, i.e. the flatness of $U(\fg)$ over $Z(\fg)$. 

     The fully faithfulness of \eqref{e:gencc2} now follows formally using the 
     analogous assertion for $Z(\fg)$ and the (tautological) conservativity of 
     the functor 
     \[  
              \on{Oblv}: D(\fg\mod_{\chi}^\heartsuit) \rightarrow 
              Z(\fg)\mod_{\chi}.  
     \]
Having shown the fully faithfulness of \eqref{e:gencc2}, it follows that all 
three categories $\scc$ appearing in \eqref{e:gencc3real} fit 
into the Cartesian square 
\[
 \xymatrix{ \scc \ar[r] \ar[d] & \fg\mod \ar[d]^{\on{Oblv}} \\ 
 Z(\fg)\mod_{\chi} \ar[r]^{i_*} & Z(\fg)\mod,}
\]
as desired. \end{proof}

%

	

	%
	%
	%
	%

	%

\subsubsection{} Having reviewed the definition of the category via infinitesimal 
characters, let us show that the two obtained categories coincide.

\begin{pro} \label{p:bgencc} The inclusions into $\fg\mod$ yield an equivalence 
	\[
	\fg\mod_{\lambda} \simeq \fg\mod_{\chi}.
	\]
\end{pro}

\begin{proof} Let us first show the inclusion 
	\[
	     \fg\mod_\lambda \subset \fg\mod_{\chi}.
	\]
   To see this, as $M_\lambda$ lies in $\fg\mod_{\chi}$ it is enough to see 
   that the latter is preserved by the 
   action of $D(G)$. This follows, for example, from the presentation 
   \eqref{e:gencc3}. 
	
     To see that the inclusion is an equality, it is enough to see that 
     $\fg\mod_{\lambda}$ contains the central quotient 
     \begin{equation}
       U(\fg) \underset{Z(\fg)} \otimes k_\chi,
     \label{e:cenq}
     \end{equation} 
	as this compactly generates $\fg\mod_{\chi}$. However, recall from \cite{beilinson-bernstein} that twisted global sections on the flag manifold, i.e. the 
	$D(G)$-equivariant functor given by convolution with $M_\lambda$
	\[
	     \Gamma^\lambda: D(G/B) \rightarrow \fg\mod_\lambda
	\]
	sends the algebra of differential operators $\sD_{G/B}$ to \eqref{e:cenq}, 
	as desired. \end{proof}

\subsection{The proof of Theorem \ref{t:linky}}

\subsubsection{} We begin by showing that the positive energy representations belong to the category for 
which we have already established a linkage principle.

\begin{pro}\label{p:poseng} The category $\per$ lies in $\hk\mod^{B\hspace{-.6mm} \on{-gen}}$.

\end{pro}

\begin{proof}  Recall that $\fh[\hspace{-.2mm}[z]\hspace{-.2mm}]$ denotes the arc algebra of $\fh$, and 
consider the similar category 
	\[
	     \PosO \subset \fh[\hspace{-.2mm}[z]\hspace{-.2mm}]\mod^\heartsuit
	\]
of $\fh[\hspace{-.2mm}[z]\hspace{-.2mm}]$-modules for which (i) the action of the Lie algebra of $K$ is 
integrable and (ii) the action of the center $Z(\fh)$ of $\sU(\fh)$ is locally finite 
with integral eigenvalues. It is straightforward to see that restriction and 
induction restrict to an adjunction 
\[
       \on{Ind}: \PosO \rightleftarrows \per: \on{Res}.
       \]

Note that the counit $\on{Ind} \circ \on{Res} \rightarrow \on{id}$ when applied 
to any object $\sM_0$ of $\per$ is surjective, i.e. yields a short exact 
sequence 
\[
   0 \rightarrow \sM_1 \rightarrow \on{Ind} \circ \on{Res}(\sM_0) \rightarrow 
   \sM_0 \rightarrow 0. 
\] 
Iterating this, we obtain a complex $\sC$
\begin{equation} \label{e:niceres}
      \cdots \rightarrow \on{Ind} \circ \on{Res} (\sM_2) \rightarrow 
    \on{Ind} 
   \circ \on{Res} (\sM_1) \rightarrow \on{Ind} \circ \on{Res} (\sM_0).
\end{equation}
For any integer $n$, let us write $\sigma^{> n} \sC$ for the corresponding 
stupid truncation 
of $\sC$.

To prove the proposition, we claim it is enough to show that the natural 
augmentation 
\begin{equation} \label{e:augeqv}
    \varinjlim_n \sigma^{> n} \sC \rightarrow \sM_0
\end{equation}
is an equivalence in $\hk\mod^K$.\footnote{Note that, due to its 
renormalization, 
$\hk\mod$ is not left-complete, which makes the assertion not entirely 
tautological.} Indeed, by its $K$-equivariance, any object $\sN$ of $\PosO$ is 
a filtered colimit of 
finite successive extensions of representations $\sN_\alpha$ inflated from 
$\fh\mod$. By our assumption on the action of $Z(\fh)$, the $\sN_\alpha$ are in fact 
inflated from $\fh\mod^{B\hspace{-.6mm} \on{-gen}}$ by Proposition \ref{p:bgencc}. Consequently, 
$\on{Ind} \sN$ is a filtered colimit of 
finite successive extensions of representations of the form 
\[
 \on{pind} (\sN_\alpha) \in \hk\mod^{B\hspace{-.6mm} \on{-gen}}. 
\]
In particular, every stupid truncation $\sigma^{> n} \sC$ visibly lies in 
$\hk\mod^{B\hspace{-.6mm} \on{-gen}}$, and hence the colimit.

It remains to show \eqref{e:augeqv} is an equivalence. 
Let us write $\fk$ for the Lie algebra of $K$. As $\hk\mod^K$ is 
compactly generated by $$\on{Ind}_\fk^{\hk} k,$$ 
it 
suffices to check this after applying (continuous) Lie algebra cohomology for 
$\fk$. Calculating this using continuous Chevalley--Eilenberg cochains, 
note the left side of \eqref{e:augeqv} produces the direct sum 
totalization of the bicomplex 
associated to
\[
   \sC \otimes \on{Sym}( \fk^*[-1]). 
\]
To see its augmentation to $\sM_0 \otimes \on{Sym}( \fk^*[-1])$ is an 
equivalence, it is enough to argue the standard associated spectral sequence is 
convergent. 

     Writing $L_0$ for 
    the 
    Segal--Sugawara energy operator, we will verify this final claim using an 
    energy estimate. Note that $L_0$ acts locally finitely on any 
    object of $\per$, and on a finitely generated object of $\per$ there exist 
    finitely many $\xi_1, \cdots, \xi_n \in k$ such that all the generalized 
    eigenvalues of $L_0$ lie in 
    \[
        \underset{i} \cup \hspace{1mm} \{ \xi_i  + \mathbb{Z}^{\geqslant 0} \} 
        \hspace{2mm}  \subset \hspace{2mm} k. 
    \]
    As the above constructions visibly commute with filtered colimits, we may 
    therefore assume that the $L_0$ eigenvalues of $\sM_0$ lie in $\xi + 
    \mathbb{Z}^{\geqslant 0}$ for some $\xi \in k$. By construction, it follows 
    that for any $n \geqslant 0$, the $L_0$ eigenvalues of $\on{Ind} \circ 
    \on{Res}(\sM_n)$ lie in $\xi + \mathbb{Z}^{\geqslant n}$, and hence the 
    same holds for 
    \[
          \on{Ind} \circ \on{Res}(\sM_n) \otimes \on{Sym}(\fk^*[-1]). 
    \]
   As the differentials in the bicomplex, and in particular the spectral 
   sequence, preserve the $L_0$ grading, the convergence follows. 
\end{proof}

\subsubsection{} We have now collected all the ingredients necessary to obtain 
Theorem \ref{t:linky}. 

\begin{proof}[Proof of Theorem \ref{t:linky}]  By Propositions \ref{p:linkp} 
and 
\ref{p:poseng}, it follows that any object $\sM$ of $\per$ canonically 
decomposes as 
\[
     \sM \simeq \underset \chi \oplus \hspace{1mm} \sM_\chi, 
\]
where $\sM_\chi$ is its summand lying in $\hk\mod_{\chi}$. It therefore remains 
to show the latter lies in $\Pos_\chi$, i.e., that for any subquotient $\sS$ of 
$\sM_\chi$, the action of $Z(\fh)$ on $\sS^K$ lies in $\chi$. To see this, 
first note 
that Propositions \ref{p:linkp} and \ref{p:poseng} imply that $\sS$ again lies 
in $\hk\mod_{\chi}$. 

Next, note that if we consider the $D(H)$-equivariant adjunction  
\[
   \on{pind}_{\fh}^{\hk}: \fh\mod \rightleftarrows \hk\mod: \on{pres}_{\hk}^{\fh},
\]
then $\sS^K$ is the zeroth cohomology of $\on{pres}_{\hk}^{\fh}(\sS)$. We will deduce the 
assertion about the $Z(\fh)$-action on $\sS^K$ from some general properties of 
$\on{pres}_{\hk}^{\fh}$.

 We first claim that 
$\on{pres}_{\hk}^{\fh}$, when restricted to $\hk\mod^{B\hspace{-.6mm} \on{-gen}}$, factors as 
\[
     \on{pres}_{\hk}^{\fh}: \hk\mod^{B\hspace{-.6mm} \on{-gen}} \rightarrow \fh\mod^{B_H\hspace{-.6mm} \on{-gen}} \subset \fh\mod. 
\]
Indeed, to see this, note that $\on{pres}_{\hk}^{\fh}$ factors through the Jacquet module, 
cf. Section \ref{ss:defpind}, i.e. 
\[
    \on{pres}_{\hk}^{\fh}: \hk\mod \rightarrow \on{pres}_{H(\!(z)\!)}^{H}(\hk\mod) \rightarrow 
    \fh\mod. 
\]
In particular, on the subcategory generated by $B$ invariants, it factors as 
\[
   \on{pres}_{\hk}^{\fh}: \hk\mod^{B\hspace{-.6mm} \on{-gen}} \rightarrow \on{pres}_{H(\!(z)\!)}^{H} (\hk\mod^{B\hspace{-.6mm} \on{-gen}}) 
   \overset{\ref{c:pindgben}} \simeq \on{pres}_{H(\!(z)\!)}^{H}(\hk\mod)^{B_H\hspace{-.6mm} \on{-gen}} \rightarrow 
   \fh\mod^{B_H\hspace{-.6mm} \on{-gen}},
\]
as desired. 

    Consider the obtained $D(H)$-equivariant adjunction 
    \[
         \on{pind}_{\fh}^{\hk}: \fh\mod^{B_H\hspace{-.6mm} \on{-gen}} \rightleftarrows \hk\mod^{B\hspace{-.6mm} \on{-gen}}: 
         \on{pres}_{\hk}^{\fh}.   
    \]
Note that if $\chi_f$ lies in $\chi$, since $\on{pind}_{\fh}^{\hk}$ sends Verma modules to 
Verma modules we deduce that $\on{pind}_{\fh}^{\hk}$ restricted to 
$\fh\mod_{\chi_f}$ factors as 
\[
       \on{pind}_{\fh}^{\hk}: \fh\mod_{\chi_f} \rightarrow \hk\mod_{\chi} \subset 
       \hk\mod^{B\hspace{-.6mm} \on{-gen}}. 
\]
Passing to right adjoints, it follows that for any $\theta$ not lying in 
$\chi$ the 
composition 
\[
    \hk\mod_{\chi} \xrightarrow{\on{pres}_{\hk}^{\fh}} \fh\mod^{B_H\hspace{-.6mm} \on{-gen}} 
    \xrightarrow{i_\theta^!} \fh\mod_{\theta}
\]
vanishes. Applying this to $\on{pres}_{\hk}^{\fh}(\sS)$ and passing to its zeroth 
cohomology yields the theorem. \end{proof}

\begin{re} Note that the theorem and its proof still hold, {\em mutatis 
mutandis}, 
after replacing the pair $H(\!(z)\!)$ and $H$ with any $G$ as in Section \ref{ss:grpslies} and a 
Levi factor $L$ of $G$.  
\end{re}

\section{Some applications}
\label{s:someapps}

In this section we develop two applications of Steinberg--Whittaker localization. Namely, we explain how to deduce some familiar equivalences between highest weight categories of representations and D-modules, and how to extend the action of the center of $U(\fg)$ on $\fg$-modules with generalized infinitesimal characters to affine type. The latter allows one to form categories of representations with `strict infinitesimal characters'.

This material will not be used in our analysis of affine Harish-Chandra bimodules, and the reader may wish to directly proceed to Section \ref{s:bimods}.

\subsection{Highest weight modules and Whittaker sheaves}  

\subsubsection{} Let $\lambda$ and $\psi$ be as in Section \ref{s:locII}, and  fix a parabolic subgroup $Q \subset G$. The following 
important theorem was proven by different means in finite type by 
Milicic--Soergel, Webster, and Backelin--Kremnizer \cite{ms}, \cite{bw}, \cite{krem}. The desire to recover it 
from a more general localization theorem on the Whittaker flag manifold was a 
motivation for the present work.

\begin{theo} There is a canonical $D(Q \backslash G / Q)$-equivariant 
equivalence
	\[
	      \fg\mod_{\lambda}^Q \simeq D(Q \backslash G / N, \psi). 
	\]
\end{theo}

\begin{proof} We simply pass to $Q$-equivariant objects in Theorem 
\ref{t:3realz}, i.e. 
	\begin{align*}
		  \fg\mod_{\lambda}^Q \simeq \on{Hom}_{D(G)\mod}(D(G/Q), 
		  \fg\mod_{\lambda}) 
	  & \simeq \on{Hom}_{D(G)\mod}(D(G/Q), D(G/N, \psi, \St)) \\ & \simeq 
	  \on{Hom}_{D(G)\mod}(D(G/Q), D(G/N, \psi)) \\ & \simeq D(Q \backslash G / 
	  N, \psi), 
	\end{align*}
	where to pass from the first to the second line one may use Remark 
	\ref{r:whitbinvs}. 
\end{proof}

\subsection{Monodromy via equivariant Hochschild cohomology}

\subsubsection{}  For a category $\sC$, recall its {\em Hochschild cohomology} 
is the endomorphisms of its identity functor, i.e. 
\[
      \on{HH}^*(\sC) := \on{Hom}_{\on{Hom}_{\DGCat}(\sC, \sC)}( \on{id}_\sC, 
      \on{id}_\sC). 
\]
Note that this would perhaps be more precisely termed its Hochschild cochains, i.e. we consider the dg-algebra itself and not simply its cohomology.

\subsubsection{} Recall that if $G$ is of finite type, there is a tautological 
map 
\[
      Z(\fg) \rightarrow \on{HH}^*(\fg\mod). 
\]
If $\lambda$ is an antidominant weight, the resulting map 
\[
    Z(\fg) \rightarrow \on{HH}^*(\fg\mod_{\lambda})
\]
factors through the completion $Z(\fg)_{\widehat{\lambda}}$ of $Z(\fg)$ at the 
infinitesimal character of $\lambda$. In particular, recalling that $\ft$ denotes the 
abstract Cartan of $G$, and $W_\lambda$ the stabilizer of $\lambda$, if we set 
\[
     R := \on{Sym}(\ft)_{\widehat{0}}
\]
we equivalently have a map 
\begin{equation} \label{e:monviaz}
     R^{W_\lambda} \simeq 
     Z(\fg)_{\widehat{\lambda}}  \rightarrow \on{HH}^*(\fg\mod_{\lambda}). 
\end{equation}

\subsubsection{} The goal of this subsection is to witness a similar action of 
`monodromy operators' in the affine case despite the triviality of the center. 
Even in finite type, this gives an formulation of these operators independent 
of the realization of the category as Lie algebra representations or as 
D-modules.

\subsubsection{} The main new player, in addition to the previous contents of 
the article, is the following. 

\begin{defn} For a $D(G)$-module $\sC$, define its {\em equivariant Hochschild 
cohomology} to be the equivariant endomorphisms of the identity functor, i.e. 
	\[
	    \on{HH}^*_G(\sC) := \on{Hom}_{\on{Hom}_{D(G)\mod}( \scc, 
	    \scc)}(\on{id}_{\sC}, \on{id}_{\sC}).  
	\]
\end{defn} 
Via the forgetful map 
\[
    \on{Hom}_{D(G)\mod}(\sC, \sC) \simeq \on{Hom}_{\DGCat}(\sC, \sC)^G 
    \xrightarrow{\on{Oblv}} \on{Hom}_{\DGCat}(\sC, \sC),
\]
there is a canonical homomorphism 
\begin{equation} \label{e:ehhtohh}
    \on{HH}^*_G(\sC) \rightarrow \on{HH}^*(\sC). 
\end{equation}

\subsubsection{} We next observe that equivariant Hochschild cohomology is 
unchanged by parabolic induction. In particular, let us fix a parabolic $P$ of 
$G$ with Levi factor $L$. 

\begin{pro}\label{p:hhpind} For a $D(L)$-module $\sC$, there is a canonical 
isomorphism 
	\[
	    \on{HH}^*_L(\sC) \simeq \on{HH}^*_G(\on{pind}_L^G(\sC))
	\]

\end{pro}
\begin{proof} For brevity, we denote $\on{pind}_L^G$ by $\on{pind}$. We first note that the unit map 
	\[
	    \sC \rightarrow \on{pind}(\sC)^{N_P}
	\]
is a fully faithful embedding. Indeed, it is obtained by tensoring over $D(L)$ 
the fully faithful left adjoint in 
\[
      i_*: D(L) \rightleftarrows D(G/N_P) : i^!.
\]	
The result then follows by considering the composition 
\[
   \on{Hom}_{D(L)\mod}( \sC, \sC) \rightarrow \on{Hom}_{D(L)\mod}( \sC, 
   \on{pind}(\sC)^{N_P}) \simeq \on{Hom}_{D(G)\mod}( \on{pind}(\sC), 
   \on{pind}(\sC)),
\]
as this is again fully faithful and sends $\on{id}_{\sC}$ to 
$\on{id}_{\on{pind}(\sC)}$. 
\end{proof}

\subsubsection{} Let us use the previous proposition to construct the monodromy 
operators. 

\begin{pro}\label{p:monops} Let $G$ and $\fg\mod_{\lambda}$ be as in Section 
\ref{s:catliereps}. There is a canonical isomorphism 
	\[
	     \on{HH}^*_G(\fg\mod_{\lambda}) \simeq 
	     R^{W_\lambda}.
	\]
	In particular, the former algebra is concentrated in cohomological degree 
	zero. 
\end{pro}

\begin{proof} By Theorem \ref{t:gmodparind} and Proposition \ref{p:hhpind} we 
have 
	\[
	     \on{HH}^*_G(\fg\mod_{\lambda}) \simeq 
	     \on{HH}^*_G(\on{pind}_L^G(\fl\mod_{\lambda})) \simeq 
	     \on{HH}^*_L(\fl\mod_{\lambda}). 
	\]
In the latter maximally singular case, by Theorem \ref{t:singloc} we have 
\[
     \on{Hom}_{D(L)\mod}( \fl\mod_{\lambda}, 
     \fl\mod_{\lambda}) \simeq \on{Hom}_{D(L)\mod}( D(L/N, -\psi), 
     \fl\mod_{\lambda}) \simeq  \fl\mod_{\lambda}^{N, \psi}. 
\]	

By construction, these equivalences send $\on{id}_{\fl\mod_{\lambda}}$ to 
\[
       i_\lambda^! \circ \on{ind}_{\fn}^{\fl} \hspace{.5mm} k_\psi,
\]	
whose endomorphisms are indeed canonically identified with $R^{W_\lambda}$. 
\end{proof}

Combining Proposition \ref{p:monops} and \eqref{e:ehhtohh} provides the desired 
extension of \eqref{e:monviaz}.

\begin{re}\label{r:strictcc}Note in particular that one may define the analogue 
of categories 
of representations with strict infinitesimal characters, namely the $D(G)$-module
\[
    \fg\mod_{\overline{\lambda}} :=  \fg\mod_{\lambda} 
    \underset{R^{W_\lambda}\mod}{\otimes} 
    \on{Vect}. 
\]
\end{re}

\begin{re} If a $D(G)$-module $\sC$ is 
dualizable, one can form its character sheaf, which is an adjoint equivariant 
D-module on the group, i.e.  
$
    \chi(\sC) \in D(G/G),
$
and one has a canonical map 
\[
\on{HH}^*_G(\sC) \rightarrow \on{Hom}_{D(G/G)}(\chi(\sC), \chi(\sC)).
\]
That is, the equivariant Hochschild cohomology of a 
$D(G)$-module $\sC$ also appears as symmetries of its character sheaf.
\end{re}

\section{Affine Harish-Chandra bimodules}
\label{s:bimods}

 In this section, we use Steinberg--Whittaker localization to obtain our main theorems on affine Harish-Chandra bimodules.


	%
	%

\subsection{Definition and first properties}

\subsubsection{} 
 To make contact with the expectations of Frenkel--Malikov, we 
now provide an analogue of 
Harish-Chandra bimodules with generalized infinitesimal characters in the affine setting. We will use the substitute for categories of $\fg$-modules with infinitesimal characters introduced in Section \ref{s:catliereps}.
%

\begin{defn}  The category of {tamely ramified} affine Harish-Chandra 
bimodules is 	
	\[
	     \HChl := (\fg\mod^{B\hspace{-.6mm} \on{-gen}} \otimes \hspace{.5mm} 
	     \fg\mod^{B\hspace{-.6mm} \on{-gen}})^G. 
	\]
\end{defn}

\begin{re}
  As is the case elsewhere in the article, since we discuss only the  
integral weights, the above category would perhaps be more precisely called 
tamely ramified with unipotent monodromy. 
\end{re}

\subsubsection{} We now discuss the basic properties of tamely ramified affine Harish-Chandra bimodules with respect to composition, i.e. semi-infinite homology. Let us denote the category of all affine Harish-Chandra bimodules (at integral, noncritical levels, cf. our definition of $\fg\mod$ in Section \ref{ss:liealgreps}) by 
\[
   \on{HCh} := (\fg\mod \otimes \hspace{.5mm} \fg\mod)^G.
\]
Recall that all bimodules $\fg\mod \otimes \hspace{.5mm} \fg\mod$, equipped with the structure of a monoidal $\infty$-category induced by semi-infinite tensoring, identifies canonically with the category of all endo-functors $\on{Hom}_{\DGCat}(\fg\mod, \fg\mod)$, cf. Section \ref{ss:liealgrepsd}. 

Passing to $G$-equivariant objects, we obtain a monoidal equivalence
\[
    \on{HCh} \simeq \on{Hom}_{D(G)\mod}(\fg\mod, \fg\mod).
\]
Let us now specialize to our subcategory of interest.

\begin{pro} The full subcategory of bimodules 
\[
    \fg\mod^{B\hspace{-.6mm} \on{-gen}} \otimes \hspace{.5mm} \fg\mod^{B\hspace{-.6mm} \on{-gen}} \subset \fg \mod \otimes \hspace{.5mm} \fg\mod 
\]
is a monoidal subcategory, i.e. is preserved by semi-infinite tensoring. Moreover, this induces monoidal equivalences 
\begin{align*}
   \fg\mod^{B\hspace{-.6mm} \on{-gen}} \otimes \hspace{.5mm} \fg\mod^{B\hspace{-.6mm} \on{-gen}} &\simeq \on{Hom}_{\DGCat}(\fg\mod^{B\hspace{-.6mm} \on{-gen}}, \fg\mod^{B\hspace{-.6mm} \on{-gen}}) \\   \HChl &\simeq \on{Hom}_{D(G)\mod}(\fg\mod^{B\hspace{-.6mm} \on{-gen}}, \fg\mod^{B\hspace{-.6mm} \on{-gen}}).\end{align*}
\end{pro}
\begin{proof} This follows from Corollary \ref{c:dualskgen}. \end{proof}

\subsection{Block decomposition}

\subsubsection{} We may decompose the above category using the following. 

\begin{defn} Let $\phi$ and $\chi$ be $W$-orbits in $\Lambda$. Define the 
associated category of affine Harish-Chandra bimodules to be 
	\[
	     \HCh_{\phi, \chi} := ( \fg\mod_{\phi} \otimes \hspace{.5mm} 
	     \fg\mod_{\chi})^G. 
	\]
\end{defn}
By a mild abuse of 
notation, let us  write $\fg\mod_{-\phi - 2 \rho}$ for 
the subcategory of $\fg\mod$ dual to $\fg\mod_{\phi}$, cf. Proposition 
\ref{p:gmodduals}.
With this, we may obtain the following direct sum decomposition. 

\begin{pro}\label{p:linkyhc} The direct sum of inclusions yields an equivalence 
	\[
	   \hspace{.5mm} \underset{\phi, 
	    \chi}\oplus \hspace{1.5mm} \hspace{.5mm}
	    \!\HCh{}_{\phi, \chi} \simeq \HChl, 
	\]
where $\phi$ and $\chi$ run through the set of all $W$-orbits on $\Lambda$. Moreover, given fixed orbits $\phi, \chi, \mu, \nu$, the composition 
\[
   \on{HCh}_{\phi, \chi} \otimes \hspace{.5mm} \on{HCh}_{\mu, \nu} \rightarrow \on{HCh}^{tame}
\]
vanishes unless $\mu = - \chi - 2 \rho$, in which case the functor factors through 
\[
   \on{HCh}_{\phi, \chi} \otimes \hspace{.5mm} \on{HCh}_{-\chi - 2 \rho, \nu} \rightarrow \on{HCh}_{\phi, \nu}. 
\]
\end{pro}

\begin{proof}This follows from Theorem \ref{t:linky} and the continuity of 
$G$ invariants, cf. Appendix D.1 of \cite{whitlocglob} for the latter. 
\end{proof}

\subsection{Blocks and affine Hecke categories}

\subsection{} Finally, we apply the localization theory developed in the previous 
sections to describe a fixed block $$\HCh_{\phi, \chi}.$$To do so, fix 
 additive characters $\psi$ and $\omega$ of $N$ corresponding to 
$\phi$ and $\chi$, respectively, cf. Section \ref{s:locII}.

\begin{theo} There are canonical equivalences \label{t:1blockhc}
	\begin{equation} \label{e:1blockhc}
	    \HCh_{\phi, \chi} \simeq \fg\mod_{\chi}^{N, \psi} \simeq \fg\mod_{-\phi 
	    - 2 \rho}^{N, -\omega} \simeq 
	    D(N, \psi, \St\hspace{-.6mm} \backslash G / N, \omega, \St).
	\end{equation}
\end{theo}

\begin{proof} Recall that for $D(G)$-modules $\scc$ and $\sD$, where $\scc$ is 
dualizable, we have a canonical equivalence 
	\begin{equation} \label{e:homduals}
	     (\scc^\vee \otimes \sD)^G \simeq \on{Hom}_{D(G)\mod}(\scc, \sD).
	\end{equation}
     Using  \eqref{e:homduals}, we have
	\begin{align} \label{e:substitutein}
	    \HCh_{\phi, \chi} \simeq ( \fg\mod_{\phi} \otimes 
	    \hspace{.5mm} 
	    \fg\mod_{\chi})^G \simeq  \on{Hom}_{D(G)\mod}(\fg\mod_{-\phi - 2 
	    \rho}, \fg\mod_{\chi}).  
	\end{align}
	By Theorem \ref{t:3realz} and Corollary \ref{c:3realz} we have canonical 
	$D(G)$-equivariant equivalences 
	\begin{equation}  \label{e:3realz}
	    \fg\mod_{-\phi - 2 \rho} \simeq D(G/N, - \psi, \St) \quad \quad 
	    \text{and} \quad \quad \fg\mod_{\chi} \simeq D(G/N, \omega, \St).
	\end{equation}
	In addition, note that for any $D(G)$-module $\sS$ generated by its 
	$B$ invariants, we have by adjunction and Corollary \ref{c:pindgben}
	\begin{align} \label{e:homstinvs}
	    &\on{Hom}_{D(G)\mod}( D(G/N, -\psi, \St), \sS) \simeq \sS^{N, \psi} \\ 
	    \label{e:homstinvs2}  
	     &\on{Hom}_{D(G)\mod}( \sS, D(G/N, \omega, \St)) 
	     \simeq 
	     \sS_{N, - \omega} \simeq \sS^{N, - \omega}. 
	\end{align}
	Substituting the equivalences of \eqref{e:3realz} into 
	\eqref{e:substitutein} and then using 
	\eqref{e:homstinvs}-\eqref{e:homstinvs2} 
	yields the four 
	desired equivalences.  
\end{proof}

We finish this subsection with a few remarks. 

\begin{re}\label{r:monoidaleq}Note that, in the special case of  $\chi = - \phi - 2 \rho$, the 
proof of Theorem \ref{t:1blockhc} gives a monoidal  equivalence 
	\[
	    \on{HCh}_{\phi, -\phi - 2\rho} \simeq D(N, \psi, \St\hspace{-.6mm} \backslash G / N, - 
	    \psi, \St)
	\]
	If $G$ is of finite type and $\phi$ a regular 
	orbit, this is essentially due to Beilinson--Ginzburg by a different 
	argument, cf. Section 5 of \cite{beilinsonginzburg} and compare with Remark 
	\ref{r:longint}. More generally, Theorem \ref{t:1blockhc} yields a monoidal equivalence 
	\begin{align*}
	   \on{HCh}^{tame} \simeq \on{Hom}_{D(G)\mod}(\fg\mod^{B\hspace{-.6mm} \on{-gen}}, \fg\mod^{B\hspace{-.6mm} \on{-gen}}) &\simeq \underset{\phi, \chi}{\oplus} \hspace{1mm} \on{Hom}_{D(G)\mod}(\fg\mod_\phi, \fg\mod_\chi) \\ & \simeq \underset{\phi, \chi} \oplus  \hspace{1mm} D(N, \psi, \St \hspace{-.6mm}\backslash G / N, \omega, \St),
	\end{align*}
i.e. of all tame Harish-Chandra bimodules with an algebroid of bi--Whittaker sheaves. 
\end{re}

\begin{re} One may use Remark \ref{r:strictcc} to construct categories of 
Harish-Chandra bimodules with `strict infinitesimal characters'. For example, if we 
set 
	\[
	   \on{HCh}_{\overline{\phi}, \chi} := (\fg\mod_{\overline{\phi}} \otimes 
	   \hspace{.5mm} \fg\mod_{\chi})^G,
	\]
	then if $\phi$ is a regular orbit a variant of the proof of Theorem 
	\ref{t:1blockhc} shows that 
	\[
	  \on{HCh}_{\overline{\phi}, \chi} \simeq \fg\mod_{\chi}^B.
	\]
	If $G$ is of finite type, this recovers an important theorem due to 
	Bernstein--Gelfand \cite{begel}. 
\end{re}

\begin{re} Suppose $\phi$ and $\chi$ are at levels of opposite sign, so that they correspond to intertwining operators between categories of $\fg$-modules at levels of the same sign. 

Under this hypothesis, Theorem \ref{t:1blockhc} exchanges an appropriately defined category of projective functors, i.e. Harish-Chandra bimodules built from tensoring with integrable modules for the affine Lie algebra, with the category of free-monodromic tilting bi-Whittaker sheaves. This, as well its relation to the Kazhdan--Lusztig fusion and some conjectures of Frenkel--Malikov, will be explained in a sequel to this paper. 
\end{re}

\bibliographystyle{amsalpha}
\bibliography{samplez}

\end{document}